\documentclass[11pt]{amsart}
\usepackage{amssymb}

\usepackage[utf8]{inputenc} 

\usepackage[colorlinks,
	citecolor=blue,
	urlcolor=blue,
	linkcolor=blue
]{hyperref} 
\newcommand{\fix}{\texorpdfstring}

\usepackage{xcolor,soul}
\usepackage{comment}

\definecolor{Crayola Bittersweet}{RGB}{254,111,94}
\definecolor{Dark Orchid}{RGB}{153,50,204}
\definecolor{Goblin Green}{RGB}{99,181,33}
\definecolor{Cobalt Ice}{RGB}{70,129,177}

\newcommand{\joe}[1]{\textcolor{Crayola Bittersweet}{Joe: #1}}

\theoremstyle{plain}
\newtheorem{thm}{Theorem}[section]
\newtheorem{lem}[thm]{Lemma}
\newtheorem{prop}[thm]{Proposition}

\theoremstyle{definition}
\newtheorem{defn}[thm]{Definition}

\newtheorem{question}[thm]{Question}

\usepackage{pgfplots}
\pgfplotsset{compat=1.7}
\usepgfplotslibrary{fillbetween}
\usetikzlibrary{patterns.meta}

\newcommand{\Worst}{\operatorname{Worst}}
\DeclareMathOperator{\uhr}{\upharpoonright}
	\renewcommand{\restriction}{\uhr}
\newcommand{\concat}{^\smallfrown}

\newcommand{\Mix}{\operatorname{Mix}}
\newcommand{\A}{\mathcal{A}}

\newcommand{\symdiff}{\mathbin{\Delta}} 
\newcommand{\cond}{\,|\,} 

\begin{document}

\title[Redundancy of information: Lowering effective dimension]{Redundancy of information:\\Lowering effective dimension}

\author[Goh]{Jun Le Goh}
\author[Miller]{Joseph S.~Miller}
\author[Soskova]{\\Mariya I.~Soskova}
\author[Westrick]{Linda Westrick}

\address[Goh]{Department of Mathematics, National University of Singapore, 10 Lower Kent Ridge Road, Singapore 119076}
\email{gohjunle@nus.edu.sg}
\urladdr{https://blog.nus.edu.sg/gohjunle/}

\address[Miller, Soskova]{Department of Mathematics, University of Wisconsin--Madison, 480 Lincoln Dr., Madison, WI 53706, USA}
\email{jmiller@math.wisc.edu}
\email{msoskova@math.wisc.edu}
\urladdr{http://www.math.wisc.edu/~jmiller/}
\urladdr{http://www.math.wisc.edu/~msoskova/}

\address[Westrick]{Penn State University Department of Mathematics, University Park, PA}
\email{westrick@psu.edu}
\urladdr{http://www.personal.psu.edu/lzw299/}

\thanks{The second and third authors are partially supported by NSF Grant No.\ DMS-2053848.  The fourth author is partially supported by NSF Grant No.\ DMS-2154173.}

\date{\today}



\begin{abstract}
Let $\mathcal A_t \subseteq 2^\omega$ denote the set of infinite sequences of effective dimension $t$. Greenberg, Miller, Shen, and Westrick~\cite{GMSW:18} studied both how near and how far an infinite sequence of dimension $s$ can be from the closest sequence of dimension $t$, where distance in $2^\omega$ is measured using the Besicovitch pseudometric. They found $\inf_{X\in \mathcal A_s} d(X,\mathcal A_t)$ and $\sup_{X\in \mathcal A_s} d(X,\mathcal A_t)$ for all $s,t\in[0,1]$, except for the supremum when $t<s<1$. This case is made difficult by the fact that the information in a dimension $s$ sequence can be coded redundantly, so it is not clear what density of changes is needed to erase enough of that information. We completely solve the dimension reduction problem.

We also identify classes of sequences for which these infima and suprema are realized as minima and maxima. When $t<s$, we find $d(X,\mathcal A_t)$ is minimized when $X$ is a Bernoulli $H^{-1}(s)$-random, and maximized when $X$ belongs to a class of infinite sequences that we call $s$-codewords. When $s<t$, the situation is reversed. Finally, we prove that all distances between the extrema are realized.
\end{abstract}

\maketitle

\section{Introduction}

The \emph{effective} (\emph{Hausdorff}) \emph{dimension} of an infinite binary sequence $X\in 2^\omega$ is
\[
\dim(X) = \liminf_{n \to \infty} \frac{K(X\restriction n)}{n},
\]
where $K$ is prefix-free Kolmogorov complexity. This notion was introduced by Lutz~\cite{L:00}. The characterization above, which was given by Mayordomo~\cite{M:02}, makes it clear what $\dim(X)$ is measuring: the asymptotic information density of $X$.

In this paper, we study how far a given infinite sequence $X\in 2^\omega$ with $\dim(X)=s$ can be from the nearest dimension $t$ sequence. In other words, how much do we need to change $X$ to rise or lower its effective dimension by a specified amount? A natural way to measure the distance between two infinite sequences $X,Y\in 2^\omega$ is to consider the \emph{Besicovitch distance}, 
\[
d(X,Y) = \limsup_{n\to\infty} \frac{\left|\{m < n: X(m) \neq Y(m)\}\right|}{n}.
\]
Note that $d$ is only a pseudometric on $2^\omega$ because $d(X,Y) = 0$ does not imply that $X$ and $Y$ are equal.
Nevertheless, if $d(X,Y)=0$, then $\dim(X) = \dim(Y)$, so each Besicovitch equivalence
class has a well-defined dimension. Furthermore, dimension and Besicovitch distance 
have the following relationship, whose proof is straightforward.

\begin{prop}[GMSW~{\cite[Proposition 3.1]{GMSW:18}}]\label{prop:lower_bound}
Consider $X,Y\in 2^\omega$ with $\dim(X)=s$ and $\dim(Y)=t$, where $s \leq t$. Then
\[
d(X,Y)\geq H^{-1}(t-s).
\]
\end{prop}
Here $H^{-1}$ is a partial inverse of the binary entropy function $H$ (see Section~\ref{sec:prelim}). So when $X$ and $Y$ are Besicovitch-close, their effective dimensions are close as well.
Of course, the converse fails very badly: $X$ and its complement $\overline{X}$ (defined by $\overline X(n) = 1-X(n)$) have the same dimension but $d(X,\overline X) = 1$.

We seek to understand in more detail the structure of the space $(2^\omega,d,\dim)$. Greenberg, Miller, Shen, and Westrick~\cite{GMSW:18} initiated the study of the relationship between effective dimension and Besicovitch distance. They showed the following. For $X \in 2^\omega$ and $t \in [0,1]$, define
\[
d(X,t) = d(X, \A_t) =\inf \{d(X,Y) : \dim(Y) = t\}.
\]

\begin{thm} Fix $s,t\in[0,1]$. Letting $X$ range over $\A_s$,
\begin{enumerate}
\item \cite[Theorem 4.1 and preceding comments on pg.\ 100]{GMSW:18} For $t\geq s$, $d(X,t)$ achieves its maximum possible value when $X$ is a Bernoulli $H^{-1}(s)$-random, and takes value $H^{-1}(t)-H^{-1}(s)$.
\item \cite[Proposition 3.5]{GMSW:18} For $t\geq s$, 
there is an $X$ for which $d(X,t)$ achieves its minimum possible value 
of $H^{-1}(t-s)$.
\item \cite[Proposition 3.5]{GMSW:18} For $t\leq s$, 
there is an $X$ for which $d(X,t)$ achieves its minimum possible value 
of $H^{-1}(s-t)$. 
\item \cite[Theorem 3.3]{GMSW:18} In the special case $s=1$, $d(X,t)$ has a constant value of $H^{-1}(1-t)$ 
(so this is also the maximum possible value of $d(X,t)$ when $s=1$).
\end{enumerate}
\end{thm}

In \cite{GMSW:18}, they asked for the proper version of the fourth item, 
considering arbitrary $t\leq s \leq 1$. Below we provide this generalization by determining the optimal upper bound for $d(X,t)$, where $t \leq s \leq 1$ and $X$ ranges over sequences of dimension $s$. With this we have a complete understanding of the minimum and maximum distance needed to move from an infinite sequence of dimension $s$ to one of dimension $t$ for arbitrary values for $s$ and $t$. Furthermore, we provide (more) details on where the minimum and maximum values of $d(X,t)$ are attained. On one hand we prove that Bernoulli randoms are as close as possible to sequences of lower dimension, complementing \cite[Theorem 4.1]{GMSW:18}. On the other hand, if the algorithmic information in a sequence $X$ is maximally redundant (in a sense made precise in Definition~\ref{def:codeword}), we call $X$ an $s$-codeword and prove that $X$ is as close as possible to sequences of higher dimension and as far as possible from sequences of lower dimension. The improved result is as follows. The first item is unchanged from \cite{GMSW:18}.

\begin{thm} \label{thm:main_thm}
Fix $s,t\in[0,1]$. Letting $X$ range over $\A_s$,
\begin{enumerate}
\item For $t\geq s$, $d(X,t)$ achieves its maximum possible value when $X$ is a Bernoulli $H^{-1}(s)$-random, and takes value $H^{-1}(t)-H^{-1}(s)$.
\item For $t\geq s$, $d(X,t)$ achieves its minimum possible value when $X$ is a 
$s$-codeword, and takes value $H^{-1}(t-s)$.
\item For $t\leq s$, $d(X,t)$ achieves its minimum possible value when $X$ is a 
Bernoulli $H^{-1}(s)$-random, and takes value $H^{-1}(s-t)$. 
\item For $t\leq s$, $d(X,t)$ achieves its maximum possible value when $X$ is an 
$s$-codeword, and takes value 
$$d(X,t) = \Worst(s,t) \mathrel{\mathop:}= \begin{cases}
H^{-1}(1-t) & \text{ if } t \leq 1 - H(c)\\
\frac{c}{s-(1-H(c))}(s-t) & \text{ if } t > 1 - H(c),
\end{cases}$$
where $c=1-2^{s-1}$.
\end{enumerate}
\end{thm}

The improvements to the second and third items are straightforward, while the fourth item is the main theorem of the present paper. In Figures~\ref{fig:decrease_dim_to_1/2} and~\ref{fig:decrease_dim_from_1/2}, we visualize items (3) and (4) from two perspectives (fixing a target dimension versus fixing a starting dimension).

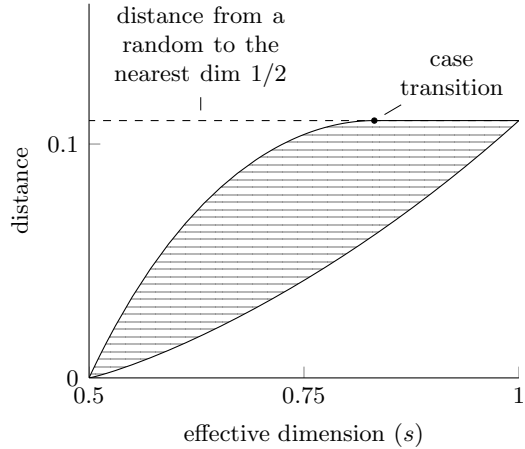
\begin{figure}[hp]
\begin{tikzpicture}
\begin{axis}[
	xlabel = {effective dimension ($s$)},
	ylabel = {distance},
	ymin = 0, ymax = 0.16, ytick = {0,0.1},
	xmin = 0.5, xmax = 1, xtick = {.5,.75,1}, 
	label style={font=\footnotesize},
	ticklabel style = {font=\footnotesize},
	x = 4.5in,
	y = 12.25in,
	axis lines* = left,
	]

	\addplot[name path = A] table[header = false] {Mathematica/Worst_s-0.5.txt};

	\addplot[name path = B] table[header = false] {Mathematica/Best_0.5.txt};

	\addplot[dashed] {0.110028} node[pos=0.563, pin={[text width=66pt, align=center, pin edge={solid, black}, pin distance=5pt] 90:{\footnotesize distance from a\\random to the\\nearest dim 1/2\\}}] {};

	\addplot[only marks, mark size = 1] coordinates {(0.831832,0.110028)} node[pin={[text width=39pt, align=center, pin edge=black, pin distance=3pt] 45:{\footnotesize case\\transition\\}}]{};

	\addplot[pattern = {Lines[angle=-45,distance={4pt/sqrt(2)}]}, pattern color=gray] fill between[of = A and B];

\end{axis}
\end{tikzpicture}
\caption{The distance from a dimension $s\geq 1/2$ sequence to the nearest dimension $1/2$ sequence. (Note that the vertical axis has been scaled up.)}
\label{fig:decrease_dim_to_1/2}
\end{figure}

\begin{figure}[hbpt]
\begin{tikzpicture}
\begin{axis}[
	xlabel = {effective dimension ($t$)},
	ylabel = {distance},
	ymin = 0, ymax = 0.5, ytick = {0,.25,.5},
	xmin = 0, xmax = 0.5, xtick = {0,.25,.5}, 
	label style={font=\footnotesize},
	ticklabel style = {font=\footnotesize},
	x = 4.5in,
	y = 4in,
	axis lines* = left
	]

	\addplot[name path = A] table[header = false] {Mathematica/Worst_0.5-s.txt};

	\addplot[name path = B] table[header = false] {Mathematica/Best_0.5.txt};

	\addplot[dashed] table[header = false] {Mathematica/Best_1.txt} node[pos=0.48, pin={[text width=64pt, align=center, pin edge={solid, black}, pin distance=8pt] 90:{\footnotesize distance from a\\random to the\\nearest dim $t$\\}}] {};

	\addplot[only marks, mark size = 1] coordinates {(0.127571,0.292893)} node[pin={[text width=39pt, align=center, pin edge=black, pin distance=3pt] 45:{\footnotesize case\\transition\\}}]{};

	\addplot[pattern = {Lines[angle=45,distance={4pt/sqrt(2)}]}, pattern color=gray] fill between[of = A and B];

\end{axis}
\end{tikzpicture}
\caption{The distance from a dimension 1/2 sequence to the nearest dimension $t\leq 1/2$ sequence.}
\label{fig:decrease_dim_from_1/2}
\end{figure}
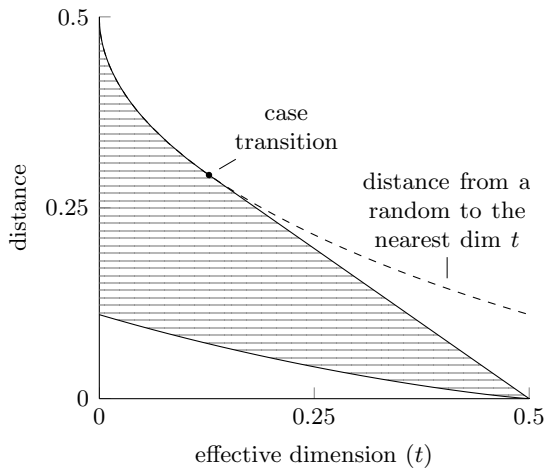

The case transition in Theorem~\ref{thm:main_thm}(4) reveals an interesting phenomenon: there is a dimension $s$ sequence $X$ (e.g., an $s$-codeword) such that if $t$ is sufficiently small compared to $s$, then $X$ is as far away from $\A_{\leq t}$ as it is possible to be, i.e., as far away as a random sequence. Such an $X$ must code its information very redundantly, at least in the sense that it is hard to erase a large amount of the information. On the other hand, thanks to the second case, we see that a very small density of changes can always erase \emph{some} information, i.e., lower the dimension at least a little. (See the discussion after Theorem~\ref{thm:dec-dim} for more on this, including how the situation for finite strings contrasts with what we have for infinite sequences.)

Having characterized the maximum and minimum possible values of $d(X,t)$ 
for $X$ of dimension $s$, it is natural to ask if all intermediate values are 
realized. We show that they are.
\begin{thm}\label{thm:intermediate} Fix $s,t \in[0,1]$.
\begin{enumerate}
\item If $t\geq s$, then for every $r \in [H^{-1}(t-s), H^{-1}(t)-H^{-1}(s)]$, there is 
an $X$ with $\dim(X) = s$ such that $d(X,t) = r$.
\item If $t \leq s$, then for every $r \in [H^{-1}(s-t), \Worst(s,t)]$, there is an $X$ 
with $\dim(X) = s$ such that $d(X,t) = r$.
\end{enumerate}
\end{thm}

These examples are obtained simply by interpolating between a 
Bernoulli $H^{-1}(s)$-random and an $s$-codeword in an appropriate 
way, but how much of each to include in the interpolation is not made 
explicit in the proof. We strengthen the above result in the special 
case when $t=1$ as follows.

\begin{thm} \label{thm:intermediate_precise}
Fix $s<1$ and $r \in [H^{-1}(1-s), 1/2-H^{-1}(s)]$. Let $Y$ be an arbitrary sequence with $\dim(Y) = 1$. Then there is a sequence $X$ with $\dim(X) = s$, $d(X,1) = r$, and $d(X,Y) = r$.
\end{thm}
Furthermore, this $X$ is obtained by a precise interpolation between 
an $s$-codeword and a Bernoulli $H^{-1}(s)$-random. It would be 
interesting to know whether all of Theorem~\ref{thm:intermediate} can 
be strengthened in a similar way, starting from a Bernoulli $H^{-1}(s)$-random
in the $t\geq s$ case, and starting from an $s$-codeword in the $t\leq s$ case.

Results relating to Theorem~\ref{thm:main_thm} have been obtained by Posobin and Shen \cite{PS:19}. They study how the dimension of a finite string or an infinite sequence is affected by random perturbations; in particular, they show that a random perturbation of a sequence (of dimension less than $1$) almost surely has higher dimension.

The proofs of the above theorems can be found as follows. Theorem~\ref{thm:main_thm}(2) is proved by combining Proposition~\ref{prop:codeword_inc-dim} and \cite[Proposition 3.5]{GMSW:18}. Theorem~\ref{thm:main_thm}(3) is proved by combining Proposition~\ref{prop:Bernoulli_dec-dim} and \cite[Proposition 3.5]{GMSW:18}. Theorem~\ref{thm:main_thm}(4) is proved in Theorem~\ref{thm:dec-dim}. Slight refinements of Theorems~\ref{thm:intermediate} and Theorem~\ref{thm:intermediate_precise} are proved in Theorems~\ref{thm:interpolate} and~\ref{thm:optimal_s_distance_dimH_1_realized} respectively.

\section{Preliminaries}
\label{sec:prelim}

We denote the set of natural numbers by $\omega$. We use \emph{sequence} for an infinite binary sequence (i.e., an element of $2^\omega$). If $X \in 2^\omega$ and $n \in \omega$, $X\restriction n$ denotes the initial segment of $X$ with length $n$. If $I$ is an interval with endpoints in $\omega$, let $|I|$ denote the number of integers in $I$. Let $X\restriction I$ denote the binary string $\sigma$ of length $|I|$ defined by $\sigma(i) = X(\min(I \cap \omega)+i)$. If $\sigma$ and $\tau$ are finite binary strings, we denote their concatenation by $\sigma \tau$ or $\sigma \concat \tau$. For $i = 0$ or $1$, $i^\omega$ is the infinite binary sequence with constant value $i$.

\subsection*{Hamming distance and the Besicovitch pseudometric}

The \emph{density} $\rho(\sigma)$ of $\sigma \in 2^n$ is the number of 1's in $\sigma$ divided by its length $n$. For $X \in 2^\omega$, the \emph{density} $\rho(X)$ is $\lim_{n \to \infty} \rho(X\restriction n)$, if the limit exists. For $\sigma,\tau \in 2^n$, the \emph{Hamming distance} $\Delta(\sigma,\tau)$ is the number of positions where $\sigma$ and $\tau$ differ, and the \emph{normalized Hamming distance} $d(\sigma,\tau)$ is $\Delta(\sigma,\tau)/n$. 

As defined above, the \emph{Besicovitch pseudometric} is given by
\[
d(X,Y) = \limsup_{n \to \infty} \frac{\Delta(X\restriction n,Y\restriction n)}{n} = \limsup_{n \to \infty} d(X\restriction n,Y\restriction n),
\]
for $X,Y \in 2^\omega$. The induced topological space is complete (see \cite{BFK:97} for a proof). Note that the associated equivalence relation (where $X$ and $Y$ are equivalent if $d(X,Y) = 0$) has been studied from the point of view of Borel equivalence relations \cite{O:06}.

\subsection*{Entropy, Hamming balls, and covering codes}

The \emph{(binary) entropy} function $H\colon [0,1] \to [0,1]$ is defined by
\[ H(p) = -p\log_2 p - (1-p)\log_2(1-p) \quad \text{and} \quad H(0) = H(1) = 0. \]
(Henceforth, we omit the base of $2$ when writing $\log$.) 
Intuitively, $H(p)$ represents the amount of information that can be encoded in a string of density $p$; by equation~\eqref{eqn:V(n,r)_basic_bound} below, there are approximately $2^{H(p)n}$ many strings in $2^n$ with density $p$. Note that $H$ is differentiable on $(0,1)$, $H(1/2) = 1$, $H$ is strictly increasing on $[0,1/2]$, and $H$ is symmetric about $p = 1/2$.

For each $\sigma \in 2^n$ and $r \leq n$, the \emph{Hamming ball $B_r(\sigma)$ with radius $r$} is the set of $\tau \in 2^n$ such that $\Delta(\sigma,\tau) \leq r$. Let $V(n,r)$ denote the size of a Hamming ball of radius $r$ in $2^n$. We note the following useful bound on $V(n,r)$ when $r\leq \frac{1}{2}n$ (see \cite[Cor.\ 9, p.\ 310]{ MS:77}): 
\begin{equation} \label{eqn:V(n,r)_basic_bound}
H\!\left(\frac{r}{n}\right)n - O(\log n)\leq \log(V(n,r))\leq H\!\left(\frac{r}{n}\right)n.
\end{equation}

By a \emph{code of length $n$} we mean a subset $C\subseteq 2^n$. We will refer to elements of $C$ as \emph{centers}, though this is not standard. For $r \leq n$, we say that the code $C\subseteq 2^n$ has \emph{covering radius at most $r$} if for every $\tau\in 2^n$, there is a $\sigma\in C$ such that $\Delta(\tau,\sigma)\leq r$, i.e., $\bigcup_{\sigma \in C} B_r(\sigma) = 2^n$. In this case, we say that $C$ is an \emph{$r$-covering code}.

Let $K(n,r)$ be the minimum size of an $r$-covering code of length $n$. There is an easy lower bound for $K(n,r)$: each Hamming ball of radius $r$ only covers $V(n,r)$ strings, so obviously $K(n,r)\geq 2^n/V(n,r)$. Delsarte and Piret~\cite{DP:86} proved the existence of fairly efficient $r$-covers; they gave an upper bound for $K(n,r)$ that is, for our purposes, quite close to the simple lower bound (see also~\cite[Theorem~12.1.2]{CHLL:97} for a proof).

\begin{thm}[Delsarte and Piret~\cite{DP:86}] \label{thm:DP}
For any $n\geq 1$ and $r\leq n$, there is an $r$-covering code $C\subseteq 2^n$ such that
\[
|C| = \left\lceil \ln(2)n\frac{2^n}{V(n,r)}\right\rceil.
\]
Therefore,
\begin{equation} \label{eqn:K(n,r)_basic_bound}
\frac{2^n}{V(n,r)}\leq K(n,r)\leq \left\lceil \ln(2)n\frac{2^n}{V(n,r)}\right\rceil.
\end{equation}
\end{thm}
The proof of this result is nonconstructive, but of course, once we know that such a cover exists, we can effectively (if not efficiently) find one via an exhaustive search.

\subsection*{Kolmogorov complexity and effective dimension}

We assume familiarity with Kolmogorov complexity for finite binary strings as well as Martin-L\"of randomness for infinite binary sequences; Downey and Hirschfeldt~\cite{DH:10} is a good reference that also includes a chapter on effective dimension.

The prefix-free (Kolmogorov) complexity of a finite string $\sigma$ is denoted by $K(\sigma)$. If $\tau$ is another finite string, we write $K(\sigma,\tau)$ for $K(\langle \sigma,\tau \rangle)$, where $\langle \cdot,\cdot \rangle\colon (2^{<\omega})^2 \to 2^{<\omega}$ is a fixed computable pairing function. The conditional complexity of $\sigma$ given $\tau$ is denoted by $K(\sigma \cond \tau)$. For an oracle $X$, let $K^X(\sigma)$ denote the prefix-free Kolmogorov complexity of $\sigma$ relative to $X$.

The \emph{dimension} of $\sigma$ is $\dim(\sigma) = K(\sigma)/|\sigma|$. For $X \in 2^\omega$, we define $\dim^X(\sigma)$ analogously. As stated previously, the \emph{(effective Hausdorff) dimension} of $X\in 2^\omega$ is $\dim(X) = \liminf_{n \to \infty} K(X\restriction n)/n$. The \emph{(effective) packing dimension} of $X$ is $\dim_P(X) = \limsup_{n \to \infty} K(X\restriction n)/n$. If $\dim(X) = \dim_P(X) = s$, we say that $X$ has \emph{dimension exactly $s$}.

Recall that for each $t \in [0,1]$, we let $\mathcal{A}_t \subseteq 2^\omega$ be the set of sequences of dimension $t$. It follows from Proposition~\ref{prop:lower_bound} that each $\mathcal{A}_t$ is closed with respect to the topology induced by $d$.

\subsection*{\boldmath Bernoulli \fix{$p$}{p}-random sequences}

For each $p \in [0,1]$, the \emph{Bernoulli $p$-measure} $\mu_p$ is defined by specifying that for $\sigma\in 2^{<\omega}$, the measure of the subbasic open set $[\sigma] = \{X\in 2^\omega : \sigma\text{ is a prefix of }X\}$ is
\[ p^{|\{n: \sigma(n) = 1\}|}(1-p)^{|\{n: \sigma(n)=0\}|}. \]
A sequence is said to be \emph{Bernoulli $p$-random} 
 if it passes all $\mu_p$-Martin-L\"of tests.  
Recall that for a probability measure $\lambda$ on $2^\omega$, a 
$\lambda$-Martin-L\"of test is a uniformly $\lambda$-c.e.\ sequence of 
open sets $\langle U_n\rangle_{n\in \omega}$ such that 
$\lambda(U_n) \leq 2^{-n}$.  Equivalently, the bound $2^{-n}$ can be 
replaced with any computable bound that approaches 0 as $n$
approaches infinity.  A sequence $X$ passes the test if 
$X \not\in \bigcap_{n \in \omega} U_n$.  So,
a sequence $X$ is not Bernoulli $p$-random if and only if there is a
$\mu_p$-c.e.\ sequence of open sets $U_n$ where 
$X \in \bigcap_{n\in \omega} U_n$ and where we have a computable 
bound on $\mu_p (U_n)$ that goes to zero in the limit.  Note that $\mu_p \equiv_T p$.\footnote{Interestingly, Kjos-Hanssen \cite{KH10} proved that if $X$ is not a Bernoulli $p$-random sequence, then $X \in \bigcap_{n\in\omega} U_n$ for a uniformly 
c.e.\ sequence $\langle U_n \rangle_n$ with $\mu_p(U_n) \leq 2^{-n}$.  That
is, the tests do not actually need the $\mu_p$ oracle in order 
to capture all non-Bernoulli-$p$-randoms.}

Martin-L\"of was the first person to consider a pointwise notion of Bernoulli $p$-randomness, in the same paper \cite{ML66} in which he defined 
 what is now called a Martin-L\"of random sequence.  We refer the reader 
 to \cite{CP19} and its references for a contemporary review. We need just a couple of facts. It is, of course, well-known that Bernoulli $p$-random sequences have
density $p$. It seems to be folklore that they have effective dimension $H(p)$, where $H$ is the entropy function defined above. For the reader's convenience we prove both facts, particularly as we 
require a slightly stronger, if routine, version of the density claim.

We use Hoeffding's inequality, a standard method to bound the 
probability of tail events in repeated Bernoulli trials. 
A special case of Hoeffding's inequality 
states that if $x_0,\dots,x_{j-1}$
is a sequence of $\{0,1\}$-valued independent random variables
with $E[x_i]=p$, then
$$P\Biggl(\Biggl|\frac{1}{j}\sum_{i<j} x_i - p\Biggr| > \varepsilon\Biggr) < 2e^{-2\varepsilon^2 j}.$$

\begin{prop}\label{prop:bernoulli-1}
Suppose that $Z$ is a Bernoulli $p$-random and let $\varepsilon>0$.
Let $K_j$ be a computable 
sequence of intervals with $|K_j| = cj \pm o(j)$, where $c>0$.
Then for all but finitely many $j$, we have that 
the density of $Z\uhr K_j$ is within $\varepsilon$ of $p$. 
In particular, $\lim_{n\rightarrow\infty} \rho(Z\uhr n) = p$.
\end{prop}

\begin{proof} It suffices to consider rational $\varepsilon$.
Let
\begin{align*}
B_j &= \{W \in 2^\omega : |\rho(W\uhr K_j) -p|>\varepsilon\}\text{, so} \\
\bigcap_{k}\bigcup_{j\geq k} B_j &= \{W \in 2^\omega : |\rho(W\uhr K_j)-p| > \varepsilon \text{ for infinitely many $j$}\}.
\end{align*}
We claim that this is a $\mu_p$-Martin-L\"of test.  The sets $\bigcup_{j\geq k} B_j$ are uniformly $p$-c.e.\ (here is where the restriction to rational 
$\varepsilon$ is used, as well as the hypothesis that the $K_j$ are 
computable).  Now we obtain a computable bound on their measure.  Hoeffding's inequality 
implies that
$$\mu_p(B_j) < 2e^{-2\varepsilon^2 |K_j|} < 2e^{-2\varepsilon^2 (c'j-M)},$$ 
where $c',M \in \mathbb Q$ are such that $0<c'<c$ and 
$|K_j| > c'j - M$ for all $j$.
Thus 
$\mu_p(\bigcup_{j\geq k} B_j)$ can be bounded above 
by $\sum_{j\geq k} 2e^{2\varepsilon^2M}(e^{-2\varepsilon^2c'})^j$, a 
geometric series whose value is a computable function of $k$.

The density of $Z$ is obtained by considering $K_j = [0,j)$.
\end{proof}

\begin{prop}\label{prop:bernoulli-2}
If $Z$ is a Bernoulli $p$-random, then
$$\dim(Z) = \dim_P(Z) = H(p).$$
\end{prop}
\begin{proof} We assume without loss of generality that $0<p<1/2$. If $p=0$ or $p=1/2$ the
result is trivial or well-known, and if $p>1/2$ we may 
simply consider $\overline Z$.

For any $\varepsilon>0$, the density of $Z$ implies that all but finitely 
many initial segments $Z\uhr n$ have density $p+\varepsilon$ or less.
There are only $2^{H(p+\varepsilon)n +O(\log n)}$ strings in $2^n$ of density $p+\varepsilon$ or less, so we have for all $\varepsilon>0$ that $K(Z\uhr n) \leq H(p+\varepsilon)n + O(\log n)$. As $\varepsilon$ was arbitrary, $\dim_P(Z) \leq H(p)$.

To prove that $\dim(Z) \geq H(p)$, given rational $\varepsilon>0$, suppose for the sake of contradiction that infinitely often $K(Z\uhr n) < (H(p)-\varepsilon)n$. 
Let rational $\delta>0$ be so small that $\delta \bigl|\log\bigl(\frac{p}{1-p}\bigr)\bigr| < \varepsilon/2$.
Let $B_n = \bigcup_{\sigma \in S_n} [\sigma]$, where
\[
S_n = \{ \sigma \in 2^n : K(\sigma) < (H(p)-\varepsilon)n \text{ and } \rho(\sigma) \geq (p-\delta)n\}.
\]
Then $|S_n| \leq 2^{(H(p)-\varepsilon)n}$. Also, the density 
requirement in the definition of $S_n$ guarantees that, for each $\sigma \in S_n$,
\[
\mu_p([\sigma]) \leq p^{(p-\delta)n}(1-p)^{(1-p+\delta)n},
\]
where here we have used the fact that $p<1/2$ to guarantee 
that $(1-p)$ is the larger number for the sake of specifying this bound. Therefore, 
\begin{align*}
\mu_p(B_n) &\leq |S_n|p^{(p-\delta)n}(1-p)^{(1-p+\delta)n}\\
 & \leq 2^{(H(p)-\varepsilon)n}2^{(p-\delta)n\log(p) + (1-p+\delta)n\log(1-p)}\\
 & = 2^{(-\varepsilon -\delta\log p + \delta \log(1-p))n}.
\end{align*}
By the choice of $\delta$, we have $\mu_p(B_n) \leq 2^{-\frac{\varepsilon}{2} n}$.  The $B_n$ are uniformly $p$-c.e.

Now the previous proposition guarantees that for all $n$
sufficiently large, $Z\uhr n$ has density at least $p-\delta$.
Therefore, $Z \in \bigcap_k \bigcup_{n\geq k} B_n$. But 
this is a $\mu_p$-Martin-L\"of test because $\mu_p(\bigcup_{n\geq k} B_n)$
is bounded above by a geometric series, similar to the 
previous proposition but simpler.
\end{proof}

\subsection*{Interpolating between infinite sequences}

Under the Besicovitch pseudometric, $2^\omega$ is pathwise connected. In other words, we can smoothly interpolate between two sequences $X_0, X_1\in 2^\omega$. Downarowicz and Iwanik \cite{DI:88} proved pathwise connectedness for the Weyl pseudometric and hence for the Besicovitch pseudometric, which never exceeds the Weyl pseudometric. A simpler proof is given in \cite{BFK:97}. We will explicitly construct an interpolation that works well for our purposes.

\begin{defn}
For $r \in [0,1]$, define $b(r) \in 2^\omega$ by recursion:
\[ b(r) = 
\begin{cases}
0^\omega & r = 0 \\
0^\omega \oplus b(2r) & 0 < r \leq 1/2 \\
b(2r-1) \oplus 1^\omega & 1/2 \leq r < 1 \\
1^\omega & r = 1,
\end{cases} \]
where $X \oplus Y$ is defined by $(X \oplus Y)(2n) = X(n)$ and $(X \oplus Y)(2n+1) = Y(n)$ for all $n \in \omega$. (This is identical to the function $f$ defined in \cite[p.~112]{BFK:97}.)
\end{defn}

Observe that $b(1/2) = 0^\omega \oplus 1^\omega$. For every $n \in \omega$, $b(r)(n)$ is defined at some finite stage of the above recursion. The recursion terminates if and only if $r$ is a dyadic rational.

\begin{lem}\label{lem:b-continuity}
$b\colon [0,1] \to 2^\omega$ is continuous with respect to the Besicovitch pseudometric $d$ on $2^\omega$.
\end{lem}
\begin{proof}
We shall prove by induction on $k$ that for every $x_0$ and $x_1$ in the dyadic interval $[m/2^k,(m+1)/2^k]$, we have $d(b(x_0),b(x_1)) \leq 2^{-k}$. It would then follow that $b$ is continuous on $[0,1]$.

The base case $k = 0$ is immediate. For the inductive step, suppose that $x_0$ and $x_1$ lie in $[m/2^{k+1},(m+1)/2^{k+1}]$. Consider first the case where this interval lies in $[0,1/2]$. Then $b(x_i) = 0^\omega \oplus b(2x_i)$ (even if $x_i = 0$ or $1/2$), and $[m/2^k,(m+1)/2^k]$ is a dyadic interval containing $2x_0$ and $2x_1$. So
\begin{align*}
d(b(x_0),b(x_1)) &= d(0^\omega \oplus b(2x_0),0^\omega \oplus b(2x_1)) \\
&= d(b(2x_0),b(2x_1))/2 \\
&\leq 2^{-k}/2,
\end{align*}
by the inductive hypothesis. The case where $[m/2^{k+1},(m+1)/2^{k+1}] \subseteq [1/2,1]$ can be handled similarly. This completes the induction.
\end{proof}

\begin{lem}
For every $r \in [0,1]$, the density of \textup{1'}s in $b(r)$ is $r$.
\end{lem}
\begin{proof}
If $r$ is dyadic, we can prove the desired result by induction. The result for general $r$ follows from the previous lemma.
\end{proof}

For each $j$, define $n_j = \sum_{i<j} i$. We refer to each interval of the form $[n_j,n_{j+1})$ as a \emph{chunk}, and denote the $j$th chunk by $I_j$. Now we define a way to interpolate between two sequences, while keeping chunks together.

\begin{defn} \label{defn:Mix}
For $X_0,X_1 \in 2^\omega$ and $r \in [0,1]$, define $\Mix(X_0,X_1,r) \in 2^\omega$ by
\[ \Mix(X_0,X_1,r) \restriction I_j = X_{b(r)(j)} \restriction I_j. \]
\end{defn}

Since $b(0)(j) = 0$ and $b(1)(j) = 1$ for all $j \in \omega$, we have $\Mix(X_0,X_1,0) = X_0$ and $\Mix(X_0,X_1,1) = X_1$.

\begin{lem} \label{lem:Mix_cts}
$\Mix(X_0,X_1,\cdot)\colon [0,1] \to 2^\omega$ is continuous with respect to the Besicovitch pseudometric $d$ on $2^\omega$.
\end{lem}
\begin{proof}
By Lemma~\ref{lem:b-continuity}, it suffices to show that the function $c\colon 2^\omega \to 2^\omega$ defined by $c(p)\restriction I_j = X_{p(j)} \restriction I_j$ is continuous with respect to the Besicovitch pseudometric. Suppose that $d(p_0,p_1)\leq \delta$. If $n\in I_{j-1}$, then we have 
\begin{align*}
\frac{\Delta(c(p_0) \restriction n, c(p_1)\restriction n)}{n} &\leq (j-1)\;\frac{\Delta(p_0 \restriction j, p_1\restriction j)}{n} \\
&\leq 2\;\frac{\Delta(p_0 \restriction j, p_1\restriction j)}{j-2},
\end{align*}
because $n\geq n_{j-1} = (j-2)(j-1)/2$. Therefore, $d(c(p_0),c(p_1)) \leq 2 \delta$.
\end{proof}

\begin{lem}\label{lem:mix-density}
For any $r \in [0,1]$, the density of \textup{1'}s in $\Mix(0^\omega,1^\omega,r)$ is $r$.
\end{lem}
\begin{proof}
We prove the result for dyadic $r$, and for the remaining $r$ 
it follows by continuity. If $r=m/2^k$ is dyadic, then the process of creating $b(r)$ terminates, and $b(r)$ is periodic with period $2^k$. Furthermore, each interval of $b(r)$ of length $2^k$ contains exactly $m$ many 1's. Now, letting $Z = \Mix(0^\omega,1^\omega,r)$, consider $Z\uhr I$ where $I$ is any interval of the form $[n_j,n_{j+2^k})$, so that this interval contains exactly $2^k$ chunks with the smallest chunk sized $j$ and the largest sized $j+2^k-1$. Letting $\ell$ be the number of 1's in $Z\uhr I$, we have 
$mj < \ell < m(j+2^k)$. We also know that $2^kj < |I| < 2^k(j+2^k)$, so
\[
\frac{mj}{2^kj+2^k} < \frac{\ell}{|I|} < \frac{m(j+2^k)}{2^kj}.
\]
Thus as $j$ increases, the density of $Z\uhr I$ approaches $r$. Furthermore,
$Z$ is a concatenation of such intervals, where each interval's length is 
negligible compared to the accumulation of intervals that came before.
Therefore, the density of $Z$ is $r$.
\end{proof}

\section{Well-distributed covering codes and \fix{$s$}{s}-codewords}
\label{sec:covering}

For our purposes, the existence of a fairly efficient $r$-covering code of length $n$ is not quite enough. Consider a Hamming ball $B_q(\tau)$ in $2^n$. How many centers from an optimal $r$-cover $C$ should we expect to be in $B_q(\tau)$? Since each $\sigma$ is in exactly $V(n,q)$ balls of radius $q$, each center has a probability of $V(n,q)/2^n$
to be in a randomly chosen Hamming ball of radius $q$. Therefore, averaging over all $\tau$, we should expect $B_q(\tau)$ to contain
\[
K(n,r)\frac{V(n,q)}{2^n} \approx \frac{V(n,q)}{V(n,r)}
\]
centers from $C$, where the approximation is idealized, but informative. Of course, this is an average, so a given ball $B_q(\tau)$ might contain more centers. We want to choose our covering codes so they do not have \emph{too many} more centers in any ball. This is because we want to be able to bound the amount of information needed to recover a center that has been moved by a Hamming distance of (at most) $q$. The following lemma tells us that sufficiently nice covering codes exist.

We will use two simple estimates for binomial coefficients. First, we have the well-known upper bound
\[
\binom{a}{b} < \bigg(\frac{ea}{b}\bigg)^b\text{, as long as }b\geq 1.
\]
Also note that as long as $a\leq b$, then
\[
\frac{a-1}{b-1}\leq \frac{a}{b}\text{, and hence }\left.\binom{a}{c}\middle/\binom{b}{c}\right.\leq \bigg(\frac{a}{b}\bigg)^c.
\]

\begin{lem}\label{lem:joe}
For any $n\geq 1$ and $0< r\leq n$\footnote{The case $r=0$ is excluded because the only $0$-covering code is $C=2^n$, which is smaller than promised by the first condition. Note that it still satisfies the second condition of the theorem.}, there is an $r$-covering code $C\subseteq 2^n$ such that
\[
|C| = S \mathrel{\mathop:}= \left\lceil \ln(2)(n+1)\frac{2^n}{V(n,r)}\right\rceil = 2^{(1-H(r/n))n+O(\log n)}.
\]
Furthermore, for every $q \in [r,n]$ and every $\tau\in 2^n$, we have
\[
|B_q(\tau)\cap C|< S_q \mathrel{\mathop:}= \left\lceil 5(n+1)\frac{V(n,q)}{V(n,r)}\right\rceil.
\]
\end{lem}
\begin{proof}
It is easy to see that the two expressions for $S$ are equal using equation~\eqref{eqn:V(n,r)_basic_bound} from Section~\ref{sec:prelim}.
We will use the probabilistic method, as is essentially used in~\cite{DP:86}, though they frame it slightly differently. The idea is to show that if we randomly choose a cover $C$ of size $S$, then with positive probability it has the desired properties.

We begin by bounding the probability that $C$ is \emph{not} an $r$-covering code. For a fixed $\tau\in 2^n$,
\[
P(\Delta(\tau,C)>r) = \left.\binom{2^n-V(n,r)}{S}\middle/\binom{2^n}{S}\right. \leq \left(1-\frac{V(n,r)}{2^n}\right)^S.
\]
Using the fact that $(1-1/x)^x < 1/e$ for all $x>1$, and by the choice of $S$,
\[
P(\Delta(\tau,C)>r) < e^{-\ln(2)(n+1)} = 2^{-n-1}.
\]
The probability of a union of events is at most the sum of their respective probabilities, hence the probability that $C$ is not an $r$-cover is
\[
P((\exists\tau)\; \Delta(\tau,C)>r) < 2^n\cdot 2^{-n-1} = 1/2.
\]

Next, we bound the probability that there are too many centers from $C$ in some Hamming ball. Fix $q\geq r$. Also fix a $\tau\in 2^n$. Note that
\begin{align*}
P(|B_q(\tau)\cap C|\geq S_q) &\leq \left.\binom{S}{S_q}\binom{V(n,q)}{S_q}\middle/\binom{2^n}{S_q}\right. \leq \binom{S}{S_q}\left(\frac{V(n,q)}{2^n}\right)^{S_q} \\
	&\leq \left(\frac{eS}{S_q}\right)^{S_q}\left(\frac{V(n,q)}{2^n}\right)^{S_q}.
\end{align*}
Let $\widehat{S}$ be the expression that we take the ceiling of to get $S$ and note that $\widehat{S}\geq 2\ln 2 > 1$. Thus we have $S < \widehat{S}+1 < \widehat{S}+\widehat{S} = 2\widehat{S}$. Also note that $S_q\geq 5(n+1)$, so
\begin{align*}
P(|B_q(\tau)\cap C|\geq S_q) &\leq \left(\frac{eS}{S_q}\cdot\frac{V(n,q)}{2^n}\right)^{S_q} < \left(\frac{2e\widehat{S}}{S_q}\cdot\frac{V(n,q)}{2^n}\right)^{S_q} \\
	&\leq \left(\frac{2e\ln(2)(n+1) 2^n}{V(n,r)}\cdot \frac{V(n,r)}{5(n+1)V(n,q)}\cdot \frac{V(n,q)}{2^n}\right)^{S_q} \\
	&= \left(\frac{2e\ln(2)(n+1)}{5(n+1)}\right)^{S_q} \leq \left(\frac{2e\ln 2}{5}\right)^{5(n+1)} < 2^{-2(n+1)},
\end{align*}
where we use the fact that $\left(2e\ln(2) / 5\right)^{5/2} < 1/2$.

Again, using the na\"\i ve bound on the probability of a union of events,
\[
P((\exists\tau)\; |B_q(\tau)\cap C|\geq S_q) < 2^n\cdot 2^{-2n-2} = 2^{-n-2}.
\]
There are $n+1$ possible values of $q$; adding the probabilities of all the ways that $C$ can fail to satisfy our requirements we have
\[
P(C\text{ does not satisfy the lemma}) < 1/2 + (n+1) 2^{-n-2},
\]
which is less than $1$ for $n\geq 1$. Therefore, a randomly chosen code $C$ has a positive probability of satisfying the lemma, and hence there exists a code that does so.
\end{proof}

Since the properties of $C$ asserted by Lemma~\ref{lem:joe} are computably recognizable, we may fix a computable sequence $\langle C^n_r \rangle_{n,r}$ of $r$-covering codes of length $n$ that satisfy the properties in Lemma~\ref{lem:joe}.

The centers of these codes $C^n_r$ have their information encoded very redundantly. For example, if $\sigma \in C^n_r$, then every $\tau$ within distance $r$ of $\sigma$ has $K(\tau) \geq K(\sigma) \pm o(n)$. This is because, by the second part of Lemma~\ref{lem:joe}, we know that $|B_r(\tau) \cap C^n_r| < 5(n+1)$. So we may give a description of $\sigma$ by giving a description of $\tau$ along with $\log\left(5(n+1)\right)$ bits to pick $\sigma$ out of $B_r(\tau)\cap C^n_r$.

Given $s \in [0,1]$, we are now ready to define a class of sequences, all of which have exact dimension $s$ and (as we shall argue in Section~\ref{sec:dec-dim}) have their information coded with a maximum of redundancy.

\begin{defn} \label{def:codeword}
Let $s \in [0,1]$. We say that $X \in 2^\omega$ is an \emph{$s$-codeword} if there is a $Y\in 2^\omega$ with $\dim(Y) = 1$ and a sequence of integers $\langle r_n\rangle_{n\in \omega}$ such that $r_n = H^{-1}(1-s)n \pm o(n)$, and such that for each $n$ we have $\Delta(X\uhr I_n, Y\uhr I_n) \leq r_n$ and $X\uhr I_n \in C^n_{r_n}$.
\end{defn}

We have the following straightforward property.
\begin{prop}\label{prop:codeword-s}
Every $s$-codeword $X$ has $\dim(X) = \dim_P(X) = s$.
\end{prop}
\begin{proof}
As the proposition is trivial when $s = 1$, we assume $s < 1$. Suppose $X$ is an $s$-codeword as witnessed by $\langle r_j\rangle_{j \in \omega}$ and $Y$. Then 
\[ d(X,Y) = \limsup_{n \to \infty}\frac{1}{n}\sum_{j<n} r_j = H^{-1}(1-s), \]
so $\dim_P(X) \geq \dim(X) \geq s$ by Proposition~\ref{prop:lower_bound}. 
Conversely, for any $j$, we may describe $X\uhr n_j$ by 
specifying, for each $i<j$, an element of $C^i_{r_i}$. By Lemma~\ref{lem:joe}, it takes $\log(2^i/V(i,r_i)) + O(\log i)$ bits to specify an element of $C^i_{r_i}$. Since $\log(V(i,r_i)) = H(r_i/i)i \pm O(\log i) = (1-s)i \pm o(i)$, the required number of bits is $si + o(i)$.
(The calculation $H(r_i/i)i = (1-s)i \pm o(i)$
follows when $s < 1$ because $H$ is differentiable on $(0,1/2]$. So when $r_i/i$ 
deviates from $H^{-1}(1-s)$ by $o(1)$, $H(r_i/i)$ 
deviates from $(1-s)$ by $o(1)$.)
We must also specify each $r_i$, since there is no computability 
restriction on that sequence, but this takes only $O(\log i)$ bits.
So we have 
\[
K(X\uhr n_j) \leq \sum_{i<j} (si + o(i)) \leq sn_j + o(j^2).
\]
Since $n_j$ grows as $j^2$, 
the second term makes no contribution in the limit, and we have 
$\dim(X) \leq s$. We also have $\dim_P(X) \leq s$, because $n_{j+1}-n_j = j+1 \ll n_j$.
\end{proof}

Now we show that starting from an $s$-codeword $X$ allows the 
dimension to be increased to $t$ using the minimum possible distance,
which is $H^{-1}(t-s)$. The strategy is to define $Y=X\symdiff Z$ where $Z$ is a Bernoulli $H^{-1}(t-s)$-random relative to $X$. 

Using the properties of Bernoulli $p$-randoms, we show that for every $s$-codeword, there is a dimension $t\geq s$ sequence at the minimum possible distance (i.e., that established by Proposition~\ref{prop:lower_bound}).

\begin{prop} \label{prop:codeword_inc-dim}
If $X$ is an $s$-codeword, then for any $t\geq s$, there is 
a sequence $Y$ with $\dim(Y) = t$ and $d(X,Y) = H^{-1}(t-s)$.
\end{prop}
\begin{proof}
First observe that if $t=1$, the proposition is trivial (using the dimension 1 sequence $Y$ which witnesses that $X$ is an $s$-codeword.) So from now on we assume that $t<1$.

Let $\langle r_j \rangle_{j \in \omega}$ be a sequence witnessing that 
$X$ is an $s$-codeword.
Let $p = H^{-1}(t-s)$, and let $Z$ be a Bernoulli $p$-random 
relative to $X$.
Let $Y = X \symdiff Z$. We claim $Y$ satisfies the proposition.
We have $d(X,Y) = p$ because $p$ is the 
density of $Z$. Therefore, $\dim(Y) \leq t$ by Proposition~\ref{prop:lower_bound}.

To show that $\dim(Y)\geq t$, fix some $\varepsilon>0$,
small enough that $p + \varepsilon < H^{-1}(1-s)$.
First we claim that
$$K(Y\uhr n_j) + o(n_j) \geq K(X\uhr n_j, Z\uhr n_j).$$
We will give a code for $X\uhr n_j$ by giving a code for 
$Y\uhr n_j$ and then concatenating $j$ codes which select 
$X\uhr n_j$ from among those elements of $C^j_{r_j}$ that
are within distance $r_j$ of $Y\uhr I_j$. 

By Lemma~\ref{lem:joe}, there are at most 
$5(j+1)$ elements of $C^j_{r_j}$ within distance $r_j$, 
so assuming that $X\uhr I_j$
is indeed among them, we can pick it out using an additional 
$o(j)$ bits. Therefore, to give all $j$ of the codes we require in total
$\sum_{i<j} o(i) \leq jo(j) \leq o(n_j)$
bits after the code for $Y\uhr I_j$.
Knowing $Y\uhr n_j$ and $X\uhr n_j$ then establishes 
the value of $Z\uhr n_j$.

Of course, if $Z\uhr I_j$ has more than $r_j$ many 1's, then
$X\uhr I_j$ would not be among those codes, so the above scheme
would fail. However,
by Proposition~\ref{prop:bernoulli-1},
for sufficiently large $j$, we have $\rho(Z\uhr I_j) < p+\varepsilon$. 
Therefore, by the choice of $\varepsilon$, for sufficiently large $j$ 
we have $\rho(Z\uhr I_j)j < r_j$. 
We may correct for finitely many $X\uhr I_j$
at the beginning using finitely much additional 
information. This proves the claim. 

Next, since $K(X\uhr n_j, Z\uhr n_j) = K(X\uhr n_j) + K(Z\uhr n_j \cond X\uhr n_j) + o(n_j)$, and since $K(Z\uhr n_j \cond X\uhr n_j) + o(n_j) \geq K^X(Z\uhr n_j)$, we know that
$$\dim(Y) \geq \dim(X) + \dim^X(Z)$$
because $\lim_j \frac{K(X\uhr n_j)}{n_j}$ and $\lim_j \frac{K^X(Z\uhr n_j)}{n_j}$
both exist (Propositions~\ref{prop:codeword-s} and~\ref{prop:bernoulli-2}, 
the latter relativized to $X$). Therefore, by those same propositions,
$$\dim(Y) \geq \dim(X) + \dim^X(Z) = s + H(p) = t.$$
as desired.
\end{proof}

\section{Decreasing dimension}
\label{sec:dec-dim}

We begin this section by proving that it is easy to decrease the dimension of a Bernoulli random. Recall from \cite[Proposition 3.1]{GMSW:18} that if $\dim(X) = s$ and $\dim(Y) = t\leq s$, then $d(X,Y) \geq H^{-1}(s-t)$. This lower bound is best possible by \cite[Proposition 3.5]{GMSW:18}. We shall improve the latter result by showing that the lower bound is realized whenever $X$ is a Bernoulli $H^{-1}(s)$-random. For this we shall need a variation of the result of Delsarte and Piret (Theorem~\ref{thm:DP}) that yields a fairly efficient $r$-cover not of all of $2^n$, but of a ball $B_q(0^n)\subseteq 2^n$ where $r \leq q \leq n$.

\begin{lem}[Vereshchagin and Vit\'anyi~{\cite[Lemma~5]{VV:10}}] \label{lem:VV}
Let $r\leq q\leq n$. Every Hamming ball of radius $q$ can be covered by at most $\alpha(n)V(n,q)/V(n,r)$ Hamming balls of radius $r$, where $\alpha(n)$ is a polynomial.
\end{lem}

\begin{prop} \label{prop:Bernoulli_dec-dim}
For any Bernoulli $H^{-1}(s)$-random $X$ and any $t<s$, there is a sequence $Y$ with $\dim(Y) = t$ and $d(X,Y) = H^{-1}(s-t)$.
\end{prop}
\begin{proof}
It suffices to show that for each $n \in \omega$, there is some $\tau \in 2^n$ such that $d(X\restriction I_n,\tau) \leq H^{-1}(s-t)$ and $\dim(\tau) \leq t + o(1)$. Given $n \in \omega$, consider the Hamming ball in $2^n$ with radius $\rho(X\restriction I_n)n$. By Lemma~\ref{lem:VV}, we can cover this ball with $\alpha(n)V(n,\rho(X\restriction I_n)n)/V(n,H^{-1}(s-t)n)$ many Hamming balls of radius $H^{-1}(s-t)n$, where $\alpha(n)$ is a polynomial. Fix a set $C$ consisting of all centers of such a set of Hamming balls. By \eqref{eqn:V(n,r)_basic_bound} in Section~\ref{sec:prelim}, we have
\begin{align*}
\log |C|  & = H(\rho(X\restriction I_n))n - H(H^{-1}(s-t))n \pm O(\log n) \\
& = H(\rho(X\restriction I_n))n - (s-t)n \pm O(\log n).
\end{align*}
Now, $H(\rho(X\restriction I_n)) = s \pm o(1)$ because $\rho(X\restriction I_n) = H^{-1}(s) \pm o(1)$ and $H$ is differentiable at $H^{-1}(s) > 0$. So $\log |C| = tn \pm o(n)$. Finally, let $\tau$ be the least string in $C$ such that $d(X\restriction I_n,\tau) \leq H^{-1}(s-t)$. By the above bound on $\log |C|$, we have $\dim(\tau) \leq t + o(1)$.
\end{proof}

In the rest of this section we address the worst case. The main result of the section, and indeed the paper, is the following.

\begin{thm}\label{thm:dec-dim} Let $0\leq t < s \leq 1$.
For any sequence $X$ of dimension $s$, there is a 
sequence $Z$ with $\dim(Z) =t$ and 
\[
d(X,Z) \leq
\begin{cases}
H^{-1}(1-t) & \text{ if } t \leq 1 - H(c)\\
\frac{c}{s-(1-H(c))}(s-t) & \text{ if } t > 1 - H(c),
\end{cases}
\]
where $c=1-2^{s-1}$.
Furthermore, these bounds are tight when $X$ is an $s$-codeword.
\end{thm}

Before giving the formal proof, we describe the general idea in a non-rigorous way (blithely omitting all lower-order error terms). Recall from \cite[Lemma 3.2]{GMSW:18} that if $\sigma$ is any initial segment of $X$, it is always possible to change $\sigma$ on bits of density at most $H^{-1}(1-t)$ to obtain a string of dimension at most $t$. In particular, this is true for $\sigma$ of dimension 1, and in this case no less than a density of $H^{-1}(1-t)$ changes will do.

The more surprising fact is that there are $\sigma$ of dimension strictly less than 1, indeed there are $\sigma$ of any dimension $s \in (0,1)$, such that reducing the dimension of $\sigma$ any further, to some $t<s$, \emph{still} requires a density of $H^{-1}(1-t)$ changes. The $\sigma$ for which this is true include all the dimension $s$ elements of any $H^{-1}(1-s)n$-covering code as in Lemma~\ref{lem:joe}. For if $\tau$ is any string with $\dim(\tau) = t$, then one can give a code for $\sigma$ by first giving a code for $\tau$ and then a code that picks $\sigma$ out of those elements of the covering code within (normalized) distance $d=d(\sigma,\tau)$ of $\tau$. By Lemma~\ref{lem:joe}, there are at most $2^{(H(d) - 1 +s)n}$ such elements of the covering code. So
\[
sn \leq tn + (H(d) -1 + s)n
\]
and rearranging we find that $d \geq H^{-1}(1-t)$, as claimed.

Thus, for finite strings, to decrease the dimension from $s$ to $t$, a density $H^{-1}(1-t)$ of changes may be required, even if $s$ is very small. The infinite case differs because, if we are given $X$ with $\sigma \prec X$ and $\dim(\sigma)=s$, we can lower the complexity of an initial segment of $X$ not by changing $\sigma$, but by changing the bits that follow it. For simplicity, first imagine that we change the next $\varepsilon n$ bits following $\sigma$ to all 0's. (Here we are imagining $\sigma$ is a difficult string as in the previous paragraph, and its length is $n$.) Then, over a total length of $(1+\varepsilon)n$, we have made at most $\varepsilon n$ changes, or a density $\varepsilon/(1+\varepsilon)$ of changes. And the dimension of the resulting string is $s/(1+\varepsilon)$. In other words, there \emph{is} an initial segment of $X$ whose dimension can be reduced below $s$ with only a small density of changes.

Now, changing the bits directly following $\sigma$ to 0 (which in most cases is a density 1/2 of changes on the post-$\sigma$ part) is not the most efficient way to lower their dimension. We could instead decide in advance that we will change density $c$ of the bits following $\sigma$, and do so in a way to decrease the dimension as much as possible. Even if the following bits have dimension 1, with density $c$ of changes we can get them down to dimension $1-H(c)$. So if our target dimension $t$ satisfies $1-H(c) < t$, then by changing the bits after $\sigma$ for long enough, we can get down to dimension $t$. (If $1-H(c) = t$, then we may have to change bits at this rate forever. This degenerate case corresponds to the original strategy from \cite[Theorem 3.3]{GMSW:18}.)

So suppose we are changing the bits after $\sigma$ at this rate $c$. If we change $cm$ of the next $m$ bits, then the new complexity of the next $m$ bits is at most $(1-H(c))m$. If we had not done anything, the complexity would have been at least $sm$ (since $\dim(X) = s$). Therefore, the benefit we got from our changes was a reduction in complexity of $(s-1+H(c))m$ bits. So if we want to get the greatest possible reduction in complexity at a minimum of changes, we would do well to choose the $c$ which maximizes the benefit/cost ratio of
\[
\frac{s-1+H(c)}{c}.
\]
In this case, the number of bits $m$ that we change should be chosen as the minimum needed to ensure that the resulting string will have dimension $t$. That is, $m$ should be chosen to satisfy
\[
sn+(1-H(c))m = t(n+m).
\]

The proof of Theorem~\ref{thm:dec-dim} uses these same ideas with more technical detail. First, we need a lemma that tells us, among other things, the optimal rate $c$ at which the bits after $\sigma$ should be changed. This $s$-dependent rate gives us the most efficient dimension reduction in the strategy described above, no matter which dimension $t$ we are aiming for. We omit the proof of the lemma which is straightforward, if tedious, calculus. See Figure~\ref{fig:understading-f} for an explanation of the geometric importance of $f(d)$.

\begin{figure}[hbpt]
\begin{tikzpicture}
\begin{axis}[
	xlabel = {$d$},
	xlabel shift = -6pt,
	ylabel = {$f(d)$},
	ylabel shift = -10pt,
	ymin = 0, ymax = 0.635777, ytick = {0,.5},
	xmin = 0, xmax = 0.5, xtick = {0,.5}, 
	label style={font=\footnotesize},
	ticklabel style = {font=\footnotesize},
	x = 4.5in,
	y = 4in,
	axis lines* = left,
	clip = false
	]

	\addplot[solid, thick] table[header = false] {Mathematica/f.txt};

	\addplot[dashed, thick] table[header = false] {Mathematica/max.txt} node[pos=0.48, pin={[text width=87pt, align=center, pin edge={solid, black}, pin distance=9pt] 0:{\footnotesize $\max\{0,s-1+H(d)\}$}}] {};

	\addplot[densely dotted, thick] table[header = false] {Mathematica/line.txt} node[pos=0.65, pin={[text width=59pt, align=center, pin edge={solid, black}, pin distance=9pt] 180:{\footnotesize $\displaystyle\frac{s-1+H(c)}{c}d$}}] {};

	\addplot[only marks, mark size = 1.25] coordinates {(0.292893,0.372429)} node[pin={[text width=60pt, align=center, pin edge=black, pin distance=3pt] -45:{\footnotesize case transition\\when $d=c$\\}}]{};

\end{axis}
\end{tikzpicture}
\caption{Illustrated for $s=1/2$, the function $f(d)$ is the smallest possible increasing, concave down function on $[0,1/2]$ that bounds $\max\{0,s-1+H(d)\}$.}
\label{fig:understading-f}
\end{figure}
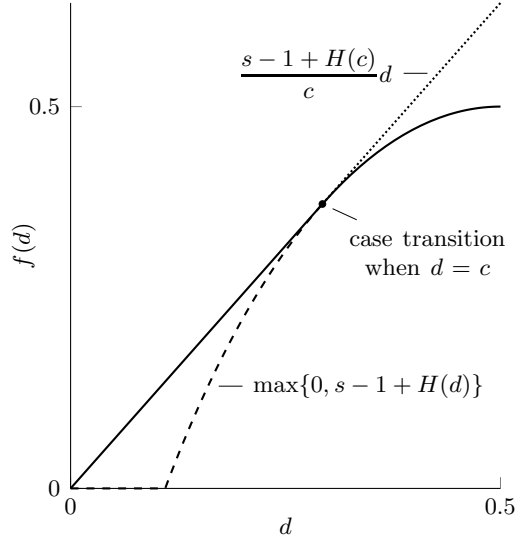

\begin{lem}\label{lem:calculus}
The function $\frac{s-1+H(d)}{d}$, 
considered as a function of $d$ on the interval $(0, 1/2]$,
attains its maximum at $c=1-2^{s-1}$. Furthermore, the function 
$f\colon [0,1/2]\rightarrow [0,s]$ defined by 
$$f(d) = \begin{cases} s-1+H(d) & \text{ if } d \geq c\\ \frac{s-1+H(c)}{c}d & \text{ if } d< c\end{cases}$$
is increasing, concave down, and satisfies 
$s-1+H(d) \leq f(d) \leq \frac{s-1+H(c)}{c}d$
on $[0,1/2]$.
\end{lem}

\begin{thm}\label{thm:exists}
For any $X$ of dimension $s$ and any $t \in (1-H(c), s)$, where $c=1-2^{s-1}$, there is a sequence $Z$ with $\dim(Z) \leq t$ and $$d(X,Z) \leq \frac{c}{s-1+H(c)}(s-t).$$
\end{thm}
\begin{proof}
Let $\ell_1, \ell_2,\dots$ be an increasing sequence of indices
such that for all $j$, $K(X\uhr \ell_j) \leq (s + 1/j)\ell_j$ and 
for all $n>\ell_j$, $K(X\uhr n) \geq (s-1/j)n$.
We require that these $\ell_j$ be rather spaced apart; in particular, we require that 
\[
\ell_{j+1} > \left\lceil\ell_j\left(1+\frac{s-t}{t-(1-H(c))}\right)\right\rceil^2.
\]
Clearly, there is a sequence $\langle \ell_j \rangle_{j \geq 1}$ satisfying these conditions.

Now, for each $j$, we define $Z\uhr[\ell_j,\ell_{j+1})$ as follows. Let
$$m_j = \left\lceil \frac{s-t}{t-(1-H(c))}\ell_j \right\rceil.$$
By the spacing condition on the 
sequence $\langle \ell_j\rangle$, we have $\ell_j + m_j < \ell_{j+1}$. (Indeed, $\ell_j + m_j < \sqrt{\ell_{j+1}}$, which will be used later.) The most naive idea would be to define $Z\uhr[\ell_j,\ell_{j+1})$ by 
changing $cm_j$ of the first $m_j$ bits of $X\uhr[\ell_j,\ell_{j+1})$ 
in some arbitrary way
in order to lower the dimension of that first part of the string 
to $1-H(c)$. There is a problem with this. If the 
$cm_j$-many changes are not distributed evenly, for example 
if they concentrated at the beginning, then 
we could end up with $d(X,Z)$ being too large. If $t$ is 
very close to $1-H(c)$, then $m_j$ is extremely large 
and density fluctuations within it would matter quite a bit. So 
we must break the interval $[\ell_j,\ell_j+m_j)$ into small chunks 
and limit the density of changes in each chunk.

We will even out the density of changes by making the changes in small chunks of 
length $\ell_{j-1}$. (Unless $j=1$, in which case we simply make the change 
in a single chunk; as all of the properties we are concerned with are tail properties, it is always safe to ignore what happens for finitely many $j$.) By Lemma~\ref{lem:joe} (with $r = \lceil c\ell_{j-1} \rceil$ and $n = \ell_{j-1}$), for each string $\sigma_k = X\uhr[\ell_j+k\ell_{j-1}, \ell_j+(k+1)\ell_{j-1})$ there is a string $\tau_k$ of complexity at
most $(1-H(c))\ell_{j-1} + o(\ell_{j-1})$
within distance $\lceil c\ell_{j-1} \rceil$ of $\sigma$. Equivalently,
$\tau_k$ has complexity at most $(1-H(c) + o(1))\ell_{j-1}$.\footnote{Whenever we use $O$- or $o$-notation in this proof, the parameters $s$ and $t$ are 
considered as constants, and the asymptotic notation bounds the 
dependence on $j$ and $\ell_j$. Thus an $o(1)$ term is a term whose limit 
is 0 as $j$ approaches infinity, while an $O(1)$ term is bounded above by a 
constant which may involve $s$ and $t$.} Let $v_j=\lfloor m_j/\ell_{j-1}\rfloor$.
Letting $\xi_j = \tau_0\tau_1\dots\tau_{v_j-1}$, we define
$$Z\uhr[\ell_j, \ell_{j+1}) = \xi_j \concat X\uhr[\ell_j+v_j\ell_{j-1}, \ell_{j+1}).$$
 
We may estimate $\dim((X\uhr \ell_j)\concat \xi_j)$
as follows:
\begin{align*}
K((X\uhr \ell_j) \concat \xi_j) &\leq s\ell_j + o(\ell_j) + \sum_{i=0}^{v_j-1} [(1-H(c))\ell_{j-1} + o(\ell_{j-1})]\\
& = s\ell_j + o(\ell_j) + (1-H(c)+o(1))m_j + o(m_j).
\end{align*}
By the definition of $m_j$, we have
$$(s-t)\ell_j = (t-(1-H(c)))m_j + O(1).$$
Therefore, since the classes $o(\ell_j)$ and $o(m_j)$ coincide, we have
\begin{alignat*}{2}
K((X\uhr \ell_j) \concat \xi_j) &\leq s\ell_j &&- (s-t)\ell_j \\
& &&+ (t-(1-H(c)))m_j + (1-H(c))m_j + o(\ell_j)\\
& = t(\ell_j+m_j) + o(\ell_j),\hidewidth
\end{alignat*}
so $\dim((X\uhr \ell_j) \concat \xi_j) \leq t +o(1)$.

We now must show that $Z$ is as required. For any $j$, 
the places where $X\uhr \ell_j$ and $Z\uhr \ell_j$ differ are confined to the interval
$[0,\ell_{j-1}+m_{j-1})$. As noted above,
$\ell_{j-1}+m_{j-1} < \sqrt{\ell_j}$. Therefore, 
$$\dim(Z\uhr \ell_j+v_j\ell_{j-1}) \leq \dim((X\uhr \ell_j)\concat \xi_j) + O(1/\sqrt{\ell_j}) \leq t+o(1).$$
This shows that $\dim Z \leq t$.

Finally, to compute the distance between $X$ and $Z$, it suffices to check 
values of $d(X\uhr n, Z\uhr n)$ where $n$ is of the form $\ell_j + k\ell_{j-1}$ for 
$k \leq (\ell_{j+1}-\ell_j)/\ell_{j-1}$. At all other $n$, the values differ by at most 
$\ell_{j-1}/\ell_j$, which is $o(\ell_j)$.
Among these points we check, it is clear the largest densities occur 
along the sequence $n=\ell_j+k \ell_{j-1}$, where $k \leq v_j$. 
In this case the number 
of bits on which $X$ and $Z$ differ is bounded by 
$\ell_{j-1}+m_{j-1}+k\lceil c\ell_{j-1}\rceil \leq ck\ell_{j-1} + o(m_j)$.
Therefore, $$d(X\uhr (\ell_j+k\ell_{j-1}), Z\uhr (\ell_j+k\ell_{j-1})) \leq \frac{ck\ell_{j-1}}{\ell_j+k\ell_{j-1}} + o(1).$$
As $x/(\ell_j+x)$ is an increasing function in $x$, and $k\ell_{j-1}\leq m_j$,
this quantity is bounded by 
$cm_j/(\ell_j+ m_j) +o(1)$. We have then, by the definition of $m_j$,
\begin{align*}
\frac{cm_j}{\ell_j+m_j} &= \frac{c \frac{s-t}{t-(1-H(c))}\ell_j}{\ell_j + \frac{s-t}{t-(1-H(c))}\ell_j} + o(1)\\
&= \frac{c(s-t)}{s-1+H(c)} + o(1),
\end{align*}
and thus $d(X,Z)$ is bounded by this quantity, as required.
\end{proof}

\begin{thm}\label{thm:forall}
Let $t<s<1$, and let $X$ be an $s$-codeword. Let $Z$ be any other sequence with $\dim Z = t$. Let $d = d(X,Z)$.
Then 
$$s - t \leq \frac{s-1+H(c)}{c} d,$$
where $c = 1-2^{s-1}$. Furthermore, if $d \geq c$, then
$$s-t \leq s-1+H(d).$$
\end{thm}
\begin{proof}
For each $i$, let $d_i = d(X\uhr I_i, Z\uhr I_i)$, so that $d_i$ 
represents the density of changes made on the $i$th chunk of $X$ when moving 
to $Z$. Let $\langle r_i\rangle_{i \in \omega}$ be a sequence witnessing that 
$X$ is an $s$-codeword. Then by Lemma~\ref{lem:joe}, we have for each $i$,
\begin{enumerate}
\item If $id_i \leq r_i$, then there are at most $O(i)$ elements of $C^i_{r_i}$ within 
distance $d_i$ from $Z \uhr I_i$, and
\item If $id_i \geq r_i$, then there are at most $2^{(H(d_i) - (1-s))i + o(i)}$ elements 
of $C^i_{r_i}$ within that distance.
\end{enumerate}
(The latter estimate also relies on the standard approximation 
$V(n,r) = 2^{H(r/n)n + o(n)}$, as well as $H(r_i/i) = H (H^{-1}(1-s) + o(1)) = 1-s + o(1),$
because $s<1$, $H$ is differentiable at $H^{-1}(1-s)$ and so the error on the 
input is within a constant factor of the error on the output.)

For each $j$, let $\bar t_j = \dim(Z\uhr[0,n_j))$, let $\bar s_j = \dim(X\uhr[0,n_j))$, and 
let $\bar d_j = d(X\uhr[0,n_j)),Z\uhr[0,n_j))$. Observe that 
$\bar d_j = \sum_{i<j} \frac{i}{n_j}d_i$. 

Now, we can always give a code for $X\uhr[0,n_j)$ 
by the following method. First give a code for $Z\uhr[0,n_j)$, 
and then for each $i<j$, give a code for $id_i$, a code for $r_i$,
and a code picking $X\uhr I_i$ out of those elements
of $C^i_{r_i}$ within distance $id_i$ of $Z\uhr I_i$.

The optimal code for $Z\uhr[0,n_j)$ has length $n_j\bar t_j$, the code for each $id_i$ and $r_i$
have length $O(\log i)$, and the length of each $C^i_{r_i}$ part has length either
$O(\log i)$ (if $id_i \leq r_i$) or $(s-1+H(d_i))i + o(i)$ (if $id_i > r_i$).
By Lemma~\ref{lem:calculus}, each 
$C^i_{r_i}$ part has length bounded by $f(d_i)i + o(i)$,
where $f$ is the function defined in that lemma.
Therefore, the code we have just described for $X\uhr[0,n_j)$ implies the following 
inequality:
$$n_j\bar s_j \leq n_j \bar t_j + \sum_{i<j} o(i) + \sum_{i<j} f(d_i)i.$$
Since $\sum_{i<j} o(i)$ is $o(j^2)$, dividing through by $n_j$ and rearranging yields
$$\bar s_j - \bar t_j \leq \sum_{i<j}\frac{i}{n_j} f(d_i) + o(1).$$
By Lemma~\ref{lem:calculus}, $f$ is concave down, so 
$$\bar s_j - \bar t_j \leq f(\bar d_j) + o(1).$$
Now we take the limsup on both sides, obtaining
$$ s - \liminf_{j\to\infty} \bar t_j \leq f\Bigr(\limsup_{j\to\infty} d_j\Bigl).$$
By Proposition~\ref{prop:codeword-s}, $\lim_{j\rightarrow \infty} \bar s_j=s$
(i.e., the packing dimension of $X$ is also $s$), which justifies the left-hand side. The right-hand
side is justified by the fact that $f$ is increasing and thus commutes with the limsup.
But this equation now says simply that 
$$s-t \leq f(d).$$
The first part of the theorem now follows because 
$f(d) \leq \frac{s-1+H(c)}{c}d$ for all $d$, and the second part follows by the 
definition of $f$.
\end{proof}

\begin{proof}[Proof of Theorem~\ref{thm:dec-dim}]
First we show that there is some $\hat Z$ with $\dim(\hat Z) \leq t$ such 
that $d(\hat X,Z)$ is as required. 
Theorem 3.3 from \cite{GMSW:18} 
states that for any $X$, we may find 
$\hat Z$ with $\dim(\hat Z) \leq t$ and $d(X,\hat Z) \leq H^{-1}(1-t)$.
By Theorem~\ref{thm:exists}, if $t > 1-H(c)$,
we can even find such 
$\hat Z$ within the smaller distance $\frac{c}{s-1+H(c)}(s-t)$, where 
$c = 1-2^{s-1}$. It is easy to see that this distance 
is in fact smaller, because when $t=1-H(c)$ both functions take the
same value $c$ and the same slope 
$-\frac{c}{s-1+H(c)}$, but $H^{-1}(1-t)$ is concave up while 
the other function is linear.

It remains to find $Z$ with $\dim(Z)= t$ and $d(X,Z)$ as required.
We shall define $Z = \Mix(X,\hat Z, r)$ (Definition~\ref{defn:Mix}) for some $r \in [0,1]$.
By the definition of $\Mix$, we have $d(X,Z) \leq d(X,\hat Z)$
regardless of what $r$ is chosen. Also, $r=0$ gives $Z=X$, which 
has dimension $s$, and 
$r=1$ gives $Z = \hat Z$, which has dimension $\leq t$. By Lemma~\ref{lem:Mix_cts}, the dimension of $Z$ varies continuously
with $r$ (this also uses the fact that $W\mapsto \dim W$ 
is continuous with respect to the Besicovitch pseudometric by
Proposition~\ref{prop:lower_bound}). 
So by the intermediate value theorem, there is some 
choice of $r$ which gives $\dim(Z) = t$, as required.

On the other hand, Theorem~\ref{thm:forall} establishes that 
when $X$ is an $s$-codeword,
the distance to any $Z$ of dimension $t$ must be at least 
$\frac{c}{s-1+H(c)}(s-t)$. This bound is obtained by a simple rearrangement
of the first bound in Theorem~\ref{thm:forall}.
If additionally $t \leq 1- H(c)$, this first
bound implies that the distance $d = d(X,Z)$ is at least $c$.
Therefore, the second bound in Theorem~\ref{thm:forall} applies,
and a simple rearrangement yields $d\geq H^{-1}(1-t)$. 
Therefore any $s$-codeword
witnesses the tightness of these bounds.
\end{proof}

\begin{question}
Say a sequence $X$ is \emph{optimally dimension $s$} if $X$ has dimension 
exactly $s$ and $d(X,1) = H^{-1}(1-s)$. Does every 
optimally dimension $s$ sequence exhibit the worst case behavior above? 
If so, does this worst case behavior characterize the optimally dimension 
$s$ sequences?
\end{question}

\section{All intermediate cases occur}
\label{sec:intermediate}

For sequences $X\in 2^\omega$ of dimension $s$, we have seen by Theorem~\ref{thm:main_thm} that
\begin{align*}
d(X,t) &\in [H^{-1}(t-s),H^{-1}(t)-H^{-1}(s)] & \text{when }t \geq s \\
\text{and}\quad d(X,t) &\in [H^{-1}(s-t),\Worst(s,t)] & \text{ when }t\leq s, 
\end{align*}
and furthermore, that one endpoint of each interval is realized when $X$ is a Bernoulli $H^{-1}(s)$-random and the other when $X$ is an $s$-codeword. We shall show that if we mix sequences of these two types, then we obtain sequences of dimension exactly $s$ that realize every intermediate value. Recall that a sequence $Y$ has dimension \emph{exactly} $s$ if $\dim(Y)=\dim_P(Y)=s$.

\begin{lem}\label{lem:mix-dimension}
Let $s \in [0,1]$, let $X$ be an $s$-codeword, and let $Z$ be 
Bernoulli $H^{-1}(s)$-random relative to $X$. Then for any $r\in [0,1]$,
the dimension of $\Mix(X,Z,r)$ is exactly $s$.
\end{lem}
\begin{proof}
We shall show that $K(\Mix(X,Z,r)\uhr n) = sn + o(n)$. Note that $K(b(r)\uhr n)$ is $O(\log n)$, so it suffices to consider $K(b(r)\uhr n, \Mix(X,Z,r)\uhr n)$ instead. To reduce the notation, we make the following abbreviations: let 
$\beta = b(r)\uhr n$ and for any $W,Y \in 2^\omega$, 
let $\alpha(W,Y) = \Mix(W,Y,r)\uhr n$. So we seek to analyze 
$K(\beta,\alpha(X,Z))$.
Since $\beta$ is enough information to determine which chunks came
from $X$ and which came from $Z$, we have 
\[
K(\beta, \alpha(X,Z)) = K(\beta, \alpha(X,0^\omega), \alpha(0^\omega,Z)) + O(1).
\]
Similarly, we can write $K(\beta, X\uhr n, Z\uhr n)$'s information as a 5-tuple
\[
K(\beta, X\uhr n, Z\uhr n) = K(\beta, \alpha(X,0^\omega), \alpha(0^\omega,Z), \alpha(Z,0^\omega), \alpha(0^\omega, X)) + O(1).
\]
As $X$ is a $s$-codeword and $Z$ is Bernoulli $H^{-1}(s)$-random relative 
to $X$, Proposition~\ref{prop:codeword-s} and the relativized 
Proposition~\ref{prop:bernoulli-2} imply that $K(X\uhr n, Z\uhr n) = 2sn + o(n)$, and thus $K(\beta, X\uhr n, Z\uhr n)$ has the same complexity up to 
an $o(n)$ term. Now we decompose the right hand side:
\begin{align*}
2sn = K(\beta) & + K(\alpha(X,0^\omega) \cond \beta)\\
& +K(\alpha(0^\omega,Z) \cond \beta,\alpha(X,0^\omega))\\
& + K(\alpha(Z,0^\omega) \cond \beta, \alpha(X,0^\omega), \alpha(0^\omega,Z))\\
& + K(\alpha(0^\omega, X) \cond \beta, \alpha(X,0^\omega), \alpha(0^\omega,Z), \alpha(Z,0^\omega)) + o(n).
\end{align*}
Since each $i$th chunk of $X$ comes from a set of codewords of log-size 
$si + o(i)$, given $\beta$ we may concatenate codes for those centers, yielding
\[
K(\alpha(X,0^\omega) \cond \beta) \leq (1-r)sn + o(n).
\]
The factor $(1-r)$ follows 
from Lemma~\ref{lem:mix-density}, because the total number of bits of
$\alpha(X,0^\omega)$ that came from $X$ is $(1-r)n + o(n)$.
By the same reasoning,
$K(\alpha(0^\omega,X) \cond \beta) \leq rsn + o(n)$.
Similarly, since each $i$th chunk of $Z$ has density $H^{-1}(s)i + o(i)$ 
(Proposition~\ref{prop:bernoulli-1}), 
each $i$th chunk of $Z$ has a code of length $si + o(i)$, and we have 
$K(\alpha(0^\omega,Z) \cond \beta) \leq rsn + o(n)$. Again, the factor of 
$r$ arises because in total $rn +o(n)$ bits 
of $\alpha(0^\omega,Z)$ were sourced from $Z$.
By the same reasoning, $K(\alpha(Z,0^\omega) \cond \beta)\leq (1-r)sn + o(n)$.
When more information is given in the conditionals, these bounds can only 
decrease further. However, their sum is $2sn +o(n)$. Therefore, each 
of these four bounds must be tight. In particular, 
$K(\alpha(X,0^\omega) \cond \beta) = (1-r)sn + o(n)$, and 
$K(\alpha(0^\omega,Z) \cond \beta, \alpha(X,0^\omega)) = rsn +o(n)$.
Adding these two terms to $K(\beta)$ gives
\[
K(\beta,\alpha(X,0^\omega),\alpha(0^\omega,Z)) = (1-r)sn + rsn + o(n) = sn +o(n),
\]
which completes the proof.
\end{proof}

\begin{thm} \label{thm:interpolate} Fix $s,t \in [0,1]$.
\begin{enumerate}
\item If $t\geq s$, then for any $d\in [H^{-1}(t-s), H^{-1}(t)-H^{-1}(s)]$, there is a sequence $X$ of dimension exactly $s$ such that $d(X,t) = d$.
\item If $t\leq s$, then for any $d \in [H^{-1}(s-t), \Worst(s,t)]$, there is a sequence $X$ of dimension exactly $s$ such that $d(X,t) = d$.
\end{enumerate}
\end{thm}
\begin{proof}
Let $X_0$ be an $s$-codeword and let $X_1$ be a Bernoulli $H^{-1}(s)$-random 
relative to $X_0$. For each $r\in[0,1]$, let $X_r = \Mix(X_0,X_1,r)$. Regardless of $r$, $\dim(X_r) = \dim_P(X_r) = s$ by Lemma~\ref{lem:mix-dimension}.
Now consider the function $r\mapsto d(X_r, t)$. By Lemma~\ref{lem:Mix_cts},
this function is continuous (using also that $Z \mapsto d(Z,t)$ is a continuous function when its domain $2^\omega$ is considered with the Besicovitch pseudometric). Now we apply the intermediate value theorem. If 
$t\geq s$, we have $d(X_0,t) = H^{-1}(t-s)$ and $d(X_1,t) = H^{-1}(t) - H^{-1}(s)$, so there is some $r$ with $d(X_r,t) = d$. If $t \leq s$, we have 
$d(X_0,t) = \Worst(s,t)$ and $d(X_1,t) = H^{-1}(s-t)$, so again there is some 
$r$ with $d(X_r,t) = d$.
\end{proof}

The above proof is not very constructive; it does not tell us---for either part of the theorem---\emph{which} sequence $X = X_r$ satisfies the claim. It also does not give any information about which (Besicovitch equivalence classes of) sequences $Y \in \A_t$ can witness that $d(X,\A_t) = d$. If $s < t$, then such sequences could be viewed as being located ``on the lower edge'' of $\A_t$, because sequences of dimension $s$ (living ``below'') encounter them first when we raise their dimension. Similarly, for $s > t$, we could consider $Y \in \A_t$ to be ``on the upper edge'' of $\A_t$ if it can serve as the witness that $d(X,\A_t) = d$ for some $X \in \A_s$. The spacial analogy seems to break down a little because we do not know whether the answer differs depending on the choice of $s$ and $d$, but we can rule out certain possibilities.

Consider first the lower edge of $\A_t$ (i.e., the case where $s<t$). If $X \in \A_s$ and $d(X,\A_t) = d < \Worst(t,s)$, then no $t$-codeword can be the witness of this distance, simply because the $d$-ball around a $t$-codeword does not intersect $\A_s$. So, at least from the standpoint of these $s$ and $d$, no $t$-codeword is on the lower edge of $\A_t$. Note that such examples occur: if $s<t<1$, then $H^{-1}(t-s) < \Worst(t,s)$, but we know that it is possible that $d(X,\A_t) = d = H^{-1}(t-s)$ for $X\in\A_s$.

Turning our attention to the upper edge, consider when $s>t$. No Bernoulli $H^{-1}(t)$-random can be on the upper edge of $\A_t$ when $d < H^{-1}(s)-H^{-1}(t)$, similarly because the $d$-ball around a Bernoulli $H^{-1}(t)$-random does not intersect $\A_s$. Again, this is not a vacuous restriction.

However, it remains possible that all Bernoulli $H^{-1}(t)$-randoms are on the lower edge of $\A_t$ (for all possible values of $s$ and $d$), and it also remains possible that all $t$-codewords are on the upper edge of $\A_t$ (again, for all possible $s$ and $d$). Formally, we have the following question.

\begin{question}\ \label{qn:interpolate}
\begin{enumerate}
\item Given $t \geq s$, $d \in [H^{-1}(t-s), H^{-1}(t)-H^{-1}(s)]$, and any Bernoulli $H^{-1}(t)$-random $Y$, is there $X\in \A_s$ with $d(X,\A_t) = d(X,Y) = d$?
\item Given $t \leq s$, $d \in [H^{-1}(s-t), \Worst(s,t)]$, and any $t$-codeword $Y$, is there $X \in \A_s$ with $d(X,\A_t) = d(X,Y) = d$?
\end{enumerate}
\end{question}

\section{Independence of distance from a random}

Here we prove a refinement of Theorem~\ref{thm:interpolate}(1) in the case $t = 1$: $X$ can be chosen to be of distance $d$ from a given sequence of dimension $1$. This answers Question~\ref{qn:interpolate}(1) in the case $t = 1$, since Bernoulli $H^{-1}(t)$-randoms have dimension $1$ when $t=1$.  First we need a lemma that follows from expressing ternary entropy in two different ways. This lemma must be well-known, but we did not find a convenient reference so we provide a proof.

\begin{lem} \label{lem:ternary_entropy_identity}
If $a \in (0,1]$, $b \in [0,1)$, and $a\geq b$, then
\[ H(a) + H\left(\frac{a-b}{a}\right)a = H(1-b) + H\left(\frac{a-b}{1-b}\right)(1-b). \]
\end{lem}
\begin{proof}
By definition of $H$,
\begin{align*}
H\left(\frac{a-b}{a}\right)a &= -a\left(\frac{a-b}{a}\right)\log\left(\frac{a-b}{a}\right) - a\left(\frac{b}{a}\right)\log\left(\frac{b}{a}\right) \\
&= -(a-b)(\log(a-b)-\log(a)) - b\log b + b\log a \\
&= -(a-b)\log(a-b) + a\log a - b\log a - b\log b + b\log a \\
&= -(a-b)\log(a-b) + a\log a - b\log b.
\end{align*}
Therefore
\begin{align*}
&H(a) + H\left(\frac{a-b}{a}\right)a \\
= \; &-a\log a - (1-a)\log(1-a) - (a-b)\log(a-b) + a\log a - b\log b \\
= \; &-(1-a)\log(1-a) - (a-b)\log(a-b) -b\log b.
\end{align*}

Observe that the above expression is invariant under the map $a \mapsto 1-b$, $b \mapsto 1-a$. (This expression is the ternary entropy with probabilities $1-a$, $a-b$ and $b$.) The invariance implies that
\begin{align*}
H(a) + H\left(\frac{a-b}{a}\right)a &= H(1-b) + H\left(\frac{1-b-(1-a)}{1-b}\right)(1-b) \\
&= H(1-b) + H\left(\frac{a-b}{1-b}\right)(1-b),
\end{align*}
as desired.
\end{proof}

\begin{thm} \label{thm:optimal_s_distance_dimH_1_realized}
Fix $s \in [0,1]$ and $d \in [H^{-1}(1-s),1/2 - H^{-1}(s)]$. For any $Y$ such that $\dim(Y) = 1$, there is a sequence $X$ of dimension exactly $s$ such that $d(X,1) = d$ and $d(X,Y) = d$ (so $Y$ witnesses that $d(X,1) \leq d$).
\end{thm}
\begin{proof}
We start by defining two sequences $X_0$ and $X_1$ that witness the theorem
when $d$ is $H^{-1}(1-s)$ or $1/2-H^{-1}(s)$, respectively. The sequence $X_0$
will be an $s$-codeword which is as close as possible to $Y$. The sequence $X_1$ will have density $H^{-1}(s)$ while being as close as possible to $Y$.

The sequence $X_0$ is constructed as follows. For each $n \in \omega$, let $r_n$ denote the integer $\lceil H^{-1}(1-s)n \rceil \leq n$. Then define $X_0\restriction I_n$ to be any element of the $r_n$-covering code $C^n_{r_n}$ which is within distance $r_n$ from $Y\restriction I_n$. By construction, $X_0$ is an $s$-codeword. So $X_0$ has dimension exactly $s$ (Proposition~\ref{prop:codeword-s}), $d(X_0,Y) \leq H^{-1}(1-s)$ (by construction), and $d(X_0,1) \geq H^{-1}(1-s)$ \cite[Proposition 3.1]{GMSW:18}. This proves the theorem in the case $d = H^{-1}(1-s)$.

The sequence $X_1$ is constructed as follows. Intuitively, we alter $Y$ to obtain $X_1$ by leaving all 0's in $Y$ unchanged, and flipping each 1 in $Y$ to a zero with some fixed probability, where the probability is chosen so that $X_1$ will have density $H^{-1}(s)$. Formally, fix some $B\in 2^\omega$ that is Bernoulli $2H^{-1}(s)$-random relative to $Y$. Listing the elements of $B = \{b_0<b_1<b_2<\cdots \}$ and $Y= \{y_0<y_1<y_2<\cdots\}$ in increasing order, we define $X_1 = \{y_{b_0},y_{b_1},y_{b_2},\dots\}$. (In \cite{JustinMiller} this would be denoted $X_1 = B \triangleright Y$). Since $Y$ has density 1/2 and $B$ has density $2H^{-1}(s)$, $X_1$ has density $H^{-1}(s)$. Therefore, by density considerations $d(X_1,1) \geq 1/2 - H^{-1}(s)$, and by construction this is also exactly $d(X_1,Y)$.

It remains to show that the dimension of $X_1$ is exactly $s$.
We have $\dim_P(X) \leq s$ because $X_1$ has density $H^{-1}(s)$.
To see that $\dim(X) \geq s$, consider the complexities 
$K(X_1\restriction n, Y\restriction n)$.
Letting $m(n)$ denote the number of 1's in $Y\restriction n$, we know that 
$m(n) = \frac{1}{2}n \pm o(n)$. Since 
$K(Y\restriction n, B\restriction m(n)) = K(X_1\restriction n, Y\restriction n) \pm o(n)$, 
we have
\[
K(Y\restriction n) + K(B\restriction m(n) \cond Y\restriction n) = K(X_1\restriction n) + K(Y\restriction n \cond X_1\restriction n) \pm o(n).
\]
As $\dim(Y) = 1$, $K(Y\restriction n) = n \pm o(n)$. Also $B$ is $2H^{-1}(s)$-random
relative to $Y$, so $K(B\restriction m(n) \cond Y\restriction n) = H(2H^{-1}(s)) m(n) \pm o(n) = \frac{1}{2}H(2H^{-1}(s)) n \pm o(n)$. Next we can place an upper bound on $K(Y\restriction n \cond X_1\restriction n)$
as follows. Letting $\ell(n)$ denote the number of 1's in $X_1\restriction n$,
we have already seen that $\ell(n) = H^{-1}(s)n \pm o(n)$. So there are 
$m(n)-\ell(n)$ many 0's in $X_1$ that become 1's in $Y$, and to specify 
$Y$ it suffices to pick these out of the total of $n-\ell(n)$ many 0's in $X_1\restriction n$.
So, letting $p = H^{-1}(s)$, 
the fraction we are picking out is 
$\frac{m(n)-\ell(n)}{n-\ell(n)} = \frac{\frac{1}{2}-p}{1-p} \pm o(1)$, 
and the total number we are picking from is $n-\ell(n) = (1-p)n \pm o(n)$,
so the number of bits required to code these 1's is $(1-p)H\left(\frac{\frac{1}{2}-p}{1-p}\right)n \pm o(n)$.
Therefore,
$$\textstyle{n + \frac{1}{2}H(2p) n \leq K(X_1\restriction n) + (1-p)H\left(\frac{\frac{1}{2}-p}{1-p}\right)n \pm o(n)},$$
which rearranges to $K(X_1\restriction n) \geq \left(1 + \frac{1}{2}H(2p) - (1-p)H\left(\frac{\frac{1}{2}-p}{1-p}\right)\right)n \pm o(n)$. Now we apply Lemma~\ref{lem:ternary_entropy_identity} with $a = \frac{1}{2}$ 
and $b=p=H^{-1}(s)$ to conclude that 
\begin{equation}\label{eqn:ternary}
\textstyle 1 + \frac{1}{2}H(2p) - (1-p)H\left(\frac{\frac{1}{2}-p}{1-p}\right) = s,
\end{equation}
as needed. (Note that this calculation uses the fact that $H(1-x) = H(x)$, i.e., the symmetry of $H$ around $1/2$.)
This proves the theorem in the case $d = 1/2-H^{-1}(s)$.

To prove the theorem when $d$ is not an endpoint, we interpolate between these
two cases. The argument is a more explicit version of what we did in Section~\ref{sec:intermediate}. Let $r$ be
$$r=\frac{d-H^{-1}(1-s)}{(1/2-H^{-1}(s))-H^{-1}(1-s)}.$$ 
This $r$, as will be seen below, is precisely the ratio in which $X_0$ 
and $X_1$ should be mixed to ensure that 
$d(X,Y) = d$. Since $d$ is not an endpoint, $0 < r < 1$.

We define both $X_0$ and $X_1$ as above, where for $X_1$ we again use a sequence $B$ that is $2H^{-1}(s)$-Bernoulli random relative to $Y$.
Let $X = \Mix(X_0,X_1,r)$.

\begin{lem}
For $X$ defined as above, $d(X,Y) \leq d$.
\end{lem}
\begin{proof} We count the bits changed as follows. 
Whenever $X\restriction I_j = X_0\restriction I_j$, we have 
$\Delta(X\restriction I_j, Y \restriction I_j) \leq r_j = H^{-1}(1-s)j + o(j)$
by the construction of $X_0$.

Also, whenever $X\restriction I_j = X_1\restriction I_j$, letting $K_j$ denote 
the interval of bits of $B$ which was consulted in the definition of $X\restriction I_j$,
we have $\Delta(X\restriction I_j, Y \restriction I_j)$ is equal to the number of 
0's in $B\restriction K_j$. We have $|K_j| = \frac{1}{2}j + o(j)$ by Proposition
\ref{prop:bernoulli-1} applied to $Y$ and the sequence of $I_j$. Applying the 
same proposition (now relativized to $Y$) to $B$ and the sequence of $K_j$,
we see that $\rho(B\restriction K_j) = 2H^{-1}(s) + o(1)$. So the number of 0's 
in $B\restriction K_j$ is 
\[
\textstyle \left(1-2H^{-1}(s)+o(1)\right)\left(\frac{1}{2}j + o(j)\right) = \left(\frac{1}{2} - H^{-1}(s)\right)j + o(j).
\]
Therefore, for each $j$,
\begin{align*}
\Delta(X\restriction n_j, Y \restriction n_j)
&\leq \sum_{\substack{i<j\\ b(r)(i)=0}} H^{-1}(1-s)i + o(i) \\[-20pt]
& \hspace{3cm} + \sum_{\substack{i<j\\ b(r)(i) = 1}}\textstyle \left(\frac{1}{2} - H^{-1}(s)\right)i + o(i)\\
&=\textstyle (1-r)H^{-1}(1-s)n_j + r\left(\frac{1}{2} - H^{-1}(s)\right)n_j + o(n_j)\\
&= dn_j + o(n_j),
\end{align*}
where the last step follows by the choice of $r$. Therefore $d(X,Y) \leq d$.
\end{proof}

The fact that $d(X,Y) \geq d$ will follow when we later show that $d(X,1) \geq d$.

\begin{lem}
For $X$ defined as above, $\dim(X) = \dim_P(X) = s$.
\end{lem}
\begin{proof} It is clear that $\dim_P(X) \leq s$
by considering how $X$ can be described on each chunk. 
Whenever $X\restriction I_j = X_0\restriction I_j$, we have 
$X\restriction I_j \in C_{r_j}^j$, so $K(X\restriction I_j) \leq sj + o(j)$.
Whenever $X \restriction I_j = X_1\restriction I_j$, we have 
$\rho(X\restriction I_j) = H^{-1}(s)j + o(j)$, and thus $K(X\restriction I_j) \leq sj + o(j)$.
By concatenating such descriptions, $K(X\restriction n_j) \leq sj + o(n_j)$.

To see that $\dim(X) \geq s$, consider $K(Y\restriction n_j, B\restriction m(n_j))$,
where as before $m(n_j)$ is the number of bits of $B$ used in the creation 
of $X_1\restriction n_j$. Similar to above, we have the following inequality:
\[
K(Y\restriction n_j) + K(B\restriction m(n_j)) \leq K(X\restriction n_j) + K(Y\restriction n_j, B\restriction m(n_j) \cond X\restriction n_j)+o(n).
\]
We can bound the right hand side by bounding, for each $j$, the quantity
$K(Y\restriction I_j, B\restriction K_j \cond X\restriction I_j)$. When 
$X\restriction I_j = X_0 \restriction I_j$, relative to $X\restriction I_j$
we can describe $Y\restriction I_j$ 
with $H(r_j)j + o(j) = (1-s)j + o(j)$ bits, and we can describe $B\restriction K_j$ 
with $\frac{1}{2}H(2p)j + o(j)$ bits (using as above that $|K_j| = \frac{1}{2}j + o(j)$ and that the density of 1's in $B\restriction K_j$ is $2p + o(1)$). When $X\restriction I_j = X_1\restriction I_j$, we can describe $Y\restriction I_j$
by describing which 0's from $X_1\restriction I_j$ should be turned to 1's. 
As we saw already, there are $(1-p)j + o(j)$ many 0's in $X_1\restriction I_j$ 
and $(\frac{1}{2}-p)j +o(j)$ many of these need to be identified and turned into 1's.
So $K(Y\restriction I_j \cond X \restriction I_j) \leq H\Bigl(\frac{\frac{1}{2}-p}{1-p}\Bigr)(1-p)j + o(j)$. Finally, given $X\restriction I_j$ and $Y\restriction I_j$, this completely 
determines $B\restriction K_j$. So we have the bound

\begin{align*}
K(Y\restriction n_j, B\restriction m(n_j) \cond X\restriction n_j)\hspace{-3cm}\\
&\leq \sum_{\substack{i<j\\b(r)(i)=0}}\textstyle \left((1-s) + \frac{1}{2}H(2p)\right)i + o(i)\\[-20pt]
&\hspace{3cm} + \sum_{\substack{i<j\\ b(r)(i)=1}}\textstyle H\left(\frac{\frac{1}{2}-p}{1-p}\right)(1-p)i + o(i)\\
 &\textstyle = (1-r)\left((1-s) + \frac{1}{2}H(2p)\right)n_j + rH\left(\frac{\frac{1}{2}-p}{1-p}\right)(1-p)n_j + o(n_j)\\
 &= \textstyle \left((1-s) + \frac{1}{2}H(2p)\right)n_j + o(n_j),
 \end{align*}
the terms related to $r$ having canceled out due to the ternary entropy identity (see equation~\eqref{eqn:ternary}).
Finally, because $K(Y\restriction n_j) = n_j + o(n_j)$ and $K(B \restriction m(n_j)) = \frac{1}{2}H(2p)n_j +o(n_j)$ we can conclude that, as required,
\begin{align*}
K(X\restriction n_j) &\textstyle \geq n_j\left(1 + \frac{1}{2}H(2p) - \left((1-s) + \frac{1}{2}H(2p)\right)\right) + o(n_j)\\
&= sn_j + o(n_j).\qedhere
\end{align*}
\end{proof}

\begin{lem}
For $X$ defined as above, $d(X,1) \geq d$.
\end{lem}
\begin{proof} Suppose $Z$ is such that $d(X,Z) < d$. We shall show that $\dim(Z) < 1$. For each $j$, let $C_j$ be the union of all $I_i$ such that $i<j$ and $b(r)(i) = 0$ (we will call the corresponding intervals $I_i$ ``center chunks'') and let $E_j$ be the union of all $I_i$ such that $i < j$ and $b(r)(i) = 1$ (we will call the corresponding $I_i$ ``erase chunks''). Note that $|C_j| + |E_j| = \sum_{i<j} |I_i| = n_j$. Writing $C_j = \{c_0<c_1<\dots<c_k\}$, let $X\restriction C_j$ 
denote $X(c_0)X(c_1)\dots X(c_k)$, and similarly for $X\restriction E_j$.

\underline{Case 1.} Suppose $Z$ is close to $X$ on many center chunks, i.e.,
\[ \liminf_{j \to \infty} d(X\restriction C_j,Z\restriction C_j) < H^{-1}(1-s). \]
Since $H$ is increasing on the range of $H^{-1}$ (i.e., $[0,1/2]$), we have
\[ \liminf_{j \to \infty} H(d(X\restriction C_j,Z\restriction C_j)) < 1-s. \]
This implies that we have infinitely many opportunities to give short descriptions for $(X\symdiff Z)\restriction C_j$, by exploiting its low density.

We shall describe $Z$ on center chunks by describing $X \symdiff Z$ ($=X_0 \symdiff Z$) on center chunks (by specifying it amongst strings with the same density) as well as $X_0$ on center chunks. As for erase chunks, we describe $Z$ verbatim. This method of description yields the following bound for each $j$:
\[ K(Z\restriction n_j) \leq |E_j| + s \cdot |C_j| + H(d(X\restriction C_j,Z\restriction C_j)) \cdot |C_j| + o(n_j). \]
To bound the right hand side, note that $\lim_j |C_j|/n_j = 1-r > 0$. Since $n_j = |E_j| + |C_j|$ and $\liminf_j H(d(X\restriction C_j,Z\restriction C_j)) < 1-s$ (as derived above), it follows that $\liminf_j K(Z\restriction n_j)/n_j < 1$. Therefore $\dim(Z) < 1$ as desired.

\underline{Case 2.} Otherwise, $Z$ is close to $X$ on many erase chunks, i.e.,
\[ \liminf_{j \to \infty} d(X\restriction E_j,Z\restriction E_j) < 1/2-H^{-1}(s). \]
This allows us to bound $\rho(Z\restriction E_j)$: Since $\lim_j \rho(X\restriction E_j) = \lim_j \rho(X_1\restriction E_j) = H^{-1}(s)$, it follows that $\liminf_j\rho(Z\restriction E_j) < 1/2$. Therefore
\[ \liminf_{j \to \infty} H(\rho(Z\restriction E_j)) < 1 \]
because $H$ is increasing on $[0,1/2]$. This implies that we have infinitely many opportunities to give short descriptions for $Z\restriction E_j$, by exploiting its low density.

Now we can describe initial segments of $Z$ by describing them verbatim on center chunks and describing them on erase chunks by specifying them among the strings of the same density. This method of description yields the following bound for each $j$:
\[ K(Z\restriction n_j) \leq |C_j| + H(\rho(Z\restriction E_j)) \cdot |E_j| + o(n_j). \]
To bound the right hand side, note that $\lim_j |E_j|/n_j = r > 0$. Since $n_j = |C_j| + |E_j|$ and $\liminf_j H(\rho(Z\restriction E_j)) < 1$ (as derived above), it follows that $\liminf_j K(Z\restriction n_j)/n_j <~1$. Therefore $\dim(Z) < 1$, as desired.
\end{proof}
We have shown that $d(X,Y) = d(X,1) = d$ and $\dim(X)=\dim_P(X) =s$,
completing the proof of Theorem~\ref{thm:optimal_s_distance_dimH_1_realized}.
\end{proof}

\bibliographystyle{plain}
\bibliography{references}

\end{document}